\documentclass[10pt]{amsart}

\usepackage{amsthm}
\usepackage{amssymb}
\usepackage{amsmath}

\numberwithin{equation}{section}
\newtheorem{thm}{Theorem}[section]
\newtheorem{prop}[thm]{Proposition}
\newtheorem{lem}[thm]{Lemma}
\newtheorem{cor}[thm]{Corollary}

\theoremstyle{definition}

\newtheorem{rem}[thm]{Remark}
\newtheorem{example}[thm]{Example}

\newcommand{\bk}{\boldsymbol{k}}

\begin{document}

\title[An algebraic identity for generating series of
        deformed MZVs]
{An identity for generating series of
        deformations of multiple zeta values within an algebraic framework }
\author{Yoshihiro Takeyama}
\address{Department of Mathematics,
        Institute of Pure and Applied Sciences,
        University of Tsukuba, Tsukuba, Ibaraki 305-8571, Japan}
\email{takeyama@math.tsukuba.ac.jp}
\thanks{This work was supported by JSPS KAKENHI Grant Number 22K03243.}

\begin{abstract}
        Bachmann proves an identity expressing the generating series
        of MacMahon's generalized sum-of-divisors $q$-series in terms of Eisenstein series.
        MacMahon's $q$-series can be regarded
        as a $q$-analogue of the multiple zeta value $\zeta(2, 2, \ldots , 2)$,
        up to a power of $1-q$.
        Based on this observation, we generalize Bachmann's identity within
        an algebraic framework and prove a general identity.
        As a byproduct, we obtain a formula for the generating series of
        another deformation of multiple zeta values defined by the author.
        In this formula, periodlike functions introduced by Lewis and Zagier appear
        as a counterpart of Eisenstein series.
\end{abstract}
%%%%%%%%%%%%%%%%%%%%%%%%%%%%%%%%%%%%%
\maketitle

\setcounter{section}{0}
\setcounter{equation}{0}

%%%%%%%%%%%%%%%%%%%%%%%%%%%%%%%%%%%%%

\section{Introduction}

\textit{Multiple zeta values} (MZVs) are real numbers defined by
\begin{align*}
        \zeta(\bk)=\sum_{0<m_{1}<\cdots <m_{r}}
        \frac{1}{m_{1}^{k_{1}} \cdots m_{r}^{k_{r}}}
\end{align*}
for a tuple of positive integers $\bk=(k_{1}, \ldots , k_{r})$
satisfying $k_{r}\ge 2$.
There are some $q$-analogues of MZVs ($q$MZVs)
(see, e.g., \cite{Zhaobook}).
The Bradley-Zhao model of $q$MZV is given by
\begin{align}
        \zeta_{q}(\bk)=\sum_{0<m_{1}<\cdots <m_{r}}
        \frac{q^{(k_{1}-1)m_{1}+\cdots +(k_{r}-1)m_{r}}}{[m_{1}]^{k_{1}} \cdots [m_{r}]^{k_{r}}},
        \label{eq:BZ-model}
\end{align}
where $0<q<1$ and $[m]=(1-q^{m})/(1-q)$ is the $q$-integer.
Since $[m]\to m$ in the limit as $q\to 1$,
we see that $\lim_{q\to 1-0}\zeta_{q}(\bk)=\zeta(\bk)$.

For $r\ge 1$, we set
\begin{align*}
        A_{r}(q)=(1-q)^{-2r}\zeta_{q}(\underbrace{2, \ldots , 2}_{r})=
        \sum_{0<m_{1}<\cdots <m_{r}}
        \frac{q^{m_{1}+\cdots +m_{r}}}{(1-q^{m_{1}})^{2} \cdots (1-q^{m_{r}})^{2}}.
\end{align*}
The $q$-series $A_{r}(q)$ is introduced by MacMahon \cite{Mac}
as a generating function of a generalization of sum-of-divisors.
To investigate its quasi-modularity, various expressions
have been derived.
In this paper, we focus on the following identity due to Bachmann \cite{Bachmann} of
the generating series of $A_{r}(q)$:
\begin{align}
        1+\sum_{r\ge 1}A_{r}(q)X^{2r}=\exp{\left(
                2\sum_{k\ge 1}\frac{(-1)^{k-1}}{(2k)!}
                G_{2k}(q)\left(2\arcsin{(X/2)}\right)^{2k}
                \right)},
        \label{eq:Bachmann}
\end{align}
where $G_{k}(q)=\sum_{n\ge 1}\sigma_{k-1}(n)q^{n}$ and
$\sigma_{a}(n)=\sum_{d\mid n}d^{a}$ is the divisor function\footnote{
        The definition of $G_{k}(q)$ differs from
        that in \cite{Bachmann}.
        Here we follow the notation in \cite{KMS}.}.
Note that $G_{2k}(q)$ is related to the Eisenstein series via the identity
\begin{align*}
        \sum_{(m, n)\in \mathbb{Z}^{2}\setminus\{(0, 0)\}}
        \frac{1}{(m+n\tau)^{2k}}=
        2\zeta(2k)+\frac{2(2\pi i)^{2k}}{(2k-1)!}G_{2k}(q),
\end{align*}
where $q=e^{2\pi i \tau}$.
See, e.g., \cite{AOS, Rose} for other expressions of $A_{r}(q)$.

It is known that the Bradley-Zhao model is closed under multiplication
if powers of $1-q$ are allowed as coefficients.
For example, we see that
\begin{align}
        \zeta_{q}(2)\zeta_{q}(3)
         & =\sum_{0<m}\frac{q^{m}}{[m]^2}\sum_{0<n}\frac{q^{2n}}{[n]^3}=
        \left(\sum_{0<m<n}+\sum_{0<n<m}+\sum_{0<m=n}\right)\frac{q^{m+2n}}{[m]^2 [n]^3}
        \label{eq:q-harmonic}                                              \\
         & =\zeta_{q}(2, 3)+\zeta_{q}(3, 2)+\zeta_{q}(5)+(1-q)\zeta_{q}(4)
        \nonumber
\end{align}
since $q^{3m}=q^{4m}+(1-q^{m})q^{3m}$ and $(1-q^{m})/[m]=1-q$.
This product structure is realized as a commutative multiplication
on a non-commutative polynomial algebra as follows.
Let $\mathcal{C}=\mathbb{Q}[\hbar, \hbar^{-1}]$ be the Laurent polynomial ring.
We endow $\mathbb{R}$ with
the structure of a $\mathcal{C}$-module
such that $\hbar^{\pm}$ acts as multiplication by $(1-q)^{\pm 1}$.
We introduce the non-commutative polynomial algebra
$\widehat{\mathfrak{H}^{1}}=\mathcal{C}\langle e_{1}, g_{1}, g_{2}, g_{3}, \ldots \rangle$.
There exists a $\mathcal{C}$-submodule $\widehat{\mathfrak{H}^{0}}$
of $\widehat{\mathfrak{H}^{1}}$ and
a $\mathcal{C}$-linear map $Z_{q}: \widehat{\mathfrak{H}^{0}} \longrightarrow \mathbb{R}$
such that $Z_{q}(1)=1$ and
$Z_{q}(e_{k_{1}}\cdots e_{k_{r}})=\zeta_{q}(k_{1}, \ldots , k_{r})$,
where $e_{k}=g_{k}+\hbar g_{k-1}$ for $k\ge 2$.
Then we can define a commutative multiplication $\ast_{\hbar}$ on $\widehat{\mathfrak{H}^{1}}$
called the \textit{harmonic product}
such that $\widehat{\mathfrak{H}^{0}}$ is closed under it and
the map $Z_{q}$ becomes a homomorphism.
For example, we have $e_{2}\ast_{\hbar}e_{3}=e_{2}e_{3}+e_{3}e_{2}+e_{5}+\hbar e_{4}$
and it implies \eqref{eq:q-harmonic}.

In the algebraic framework,
the $q$-series $A_{r}(q)$ is equal to $Z_{q}(\hbar^{-2r}e_{2}^{r})$,
and there exists an element $\phi_{k}$ of $\widehat{\mathfrak{H}^{0}}$
such that $Z_{q}(\hbar^{-k}\phi_{k})=G_{k}(q)$ for $k\ge 1$.
Then Bachmann's identity \eqref{eq:Bachmann} follows from the equality
\begin{align}
        \frac{1}{1-\hbar^{-2}e_{2}X^2}=\exp_{\ast_{\hbar}}{\left(
                2\sum_{k\ge 1}\frac{(-1)^{k-1}}{(2k)!} \hbar^{-2k}\phi_{2k}
                (2\arcsin{(X/2)})^{2k}
                \right)}
        \label{eq:Bachmann-algebraic}
\end{align}
in the formal power series ring $\widehat{\mathfrak{H}^{0}}[[X]]$,
where $\exp_{\ast_{\hbar}}$ is the exponential with respect to the harmonic product
(see Section \ref{sec:harmonic-product} for the definition).

The main purpose of this paper is to prove a general identity
which contains \eqref{eq:Bachmann-algebraic} as a special case
(see Theorem \ref{thm:main1} below).
It gives an expression of a formal power series of the form
$1/(1-pX)$ for any element $p$ of the $\mathbb{Q}$-module
$\sum_{k\ge 2}\mathbb{Q}\,\hbar^{-k}e_{k}$
in terms of the exponential $\exp_{\ast_{\hbar}}$ whose exponent is
a linear combination of the elements $\hbar^{-k}\phi_{k} \, (k\ge 1)$.
In the general identity,
$\arcsin{(X/2)}$ in \eqref{eq:Bachmann-algebraic} is replaced with
a convergent power series of the form $\mathrm{Log}{(1-\alpha(X))}$.
Here $\alpha(x)$ is a function such that
$z=1/\alpha(x)$ is a root of $1-x^{N}P(z)$,
where $P(z)$ is a polynomial with rational coefficients which is
determined from $p$, and $N=\deg{P}$.
Although it seems difficult to compute $\mathrm{Log}(1-\alpha(X))$ in general,
it is possible to carry out the calculation in the case where
$P(z)=z^{N-1}(z-1)$, which corresponds to $p=(-\hbar)^{-N}e_{N}$,
with $N\ge 2$ (see Theorem \ref{thm:main2} below).
Then we obtain Bachmann's identity \eqref{eq:Bachmann} in the case of $N=2$.

Our general identity is an equality of formal power series with coefficients
in $\widehat{\mathfrak{H}^{0}}$.
Hence,
we can obtain an equality by applying any homomorphism
with respect to the harmonic product
to both sides of the general identity.
For example, we get the identity due to
Kang, Matsusaka and Shin \cite{KMS},
which is a generalization of \eqref{eq:Bachmann},
{}from our general identity (see Remark \ref{rem:KMS} below).
In this paper we discuss another example.
In \cite{T-omega},
the author introduces a deformation of MZVs called
the \textit{$\omega$-deformed MZV} ($\omega$MZV).
It contains a positive parameter $\omega$ and
we recover MZV in the limit as $\omega \to +0$.
The $\omega$MZVs are closed under multiplication,
and their product structure is the same as $q$MZVs
where $1-q$ in the coefficients are replaced with $2\pi i \omega$.
In other words, there exists a homomorphism, by abuse of notation,
$Z_{\omega}: \widehat{\mathfrak{H}^{0}} \longrightarrow \mathbb{C}$
with respect to the harmonic product $\ast_{\hbar}$,
where $\mathbb{C}$ is endowed with the structure of a $\mathcal{C}$-module
such that $\hbar^{\pm 1}$ acts as multiplication by $(2\pi i \omega)^{\pm 1}$.
Hence we can obtain a formula for a generating series of $\omega$MZVs from
our general identity by applying the map $Z_{\omega}$.
For example, we find that
\begin{align*}
        1+\sum_{r\ge 1}Z_{\omega}(e_{2}^{r})X^{2r}=
        \frac{\sin{(\omega^{-1}\arcsin{(\pi i \omega X)})}}{\pi i X}
        \exp{\left(-\frac{\arcsin^{2}{(\pi i \omega X)}}{\pi i \omega}\right)}
\end{align*}
(see Corollary \ref{cor:omega-MZV-222} below).
By taking the limit as $\omega \to +0$,
we recover the formula for MZVs
\begin{align*}
        1+\sum_{r\ge 1}\zeta(\underbrace{2, \ldots , 2}_{r})X^{2r}=
        \frac{\sin{\pi i X}}{\pi i X}=
        \sum_{r\ge 0}\frac{\pi^{2r}}{(2r+1)!}X^{2r}.
\end{align*}

Finally we remark that the image $Z_{\omega}(\hbar^{-k}\phi_{k})$,
which is a counterpart of the $q$-series
$Z_{q}(\hbar^{-k}\phi_{k})=G_{k}(q)=\sum_{n\ge 1}\sigma_{k-1}(n)q^{n}$, is a function of $\omega$
satisfying
\begin{align*}
        f(\omega)=f(\omega+1)+(\omega+1)^{-2s}f(\frac{\omega}{\omega+1})
\end{align*}
with $2s=k$.
In \cite{LZ} Lewis and Zagier study general solutions
of the above three-term functional equation,
which they call \textit{periodlike functions},
in connection with Maass wave forms.

We outline the structure of this paper below.
In Section \ref{sec:algebra} we describe the algebraic framework for
$q$MZVs in more detail.
In Section \ref{sec:main} we prove the general identity.
In Section \ref{sec:solvable} we consider the solvable case where
$P(z)=z^{N-1}(z-1)$ and show how Bachmann's identity \eqref{eq:Bachmann} is derived from
the general identity.
In Section \ref{sec:omega-MZV} we review the definition of
the $\omega$MZV and its properties,
and derive a formula for a generating series of $\omega$MZVs from
the general identity.
Appendix \ref{sec:app-stirling} is devoted to a brief review of Stirling numbers.
In Appendix \ref{sec:app2}
we prove some properties of the power series needed to express the roots of
the polynomial $1-x^{N}P(z)$ in the solvable case where
$P(z)=z^{N-1}(z-1)$ mentioned above.

Throughout this paper we denote by
${m\brack n}$
the Stirling number of the first kind
and by ${m\brace n}$
that of the second kind.

%%%%%%%%%%%%%%%%%%%%%%%%%%%%%%%%%%%%%

\section{Algebraic framework for $q$MZVs}\label{sec:algebra}

\subsection{Preliminaries}

Let $\mathcal{C}=\mathbb{Q}[\hbar, \hbar^{-1}]$ be
the Laurent polynomial ring of the formal variable $\hbar$ over $\mathbb{Q}$.
We denote by $\widehat{\mathfrak{H}^{1}}$
the non-commutative polynomial ring over $\mathcal{C}$
with variables $e_{1}$ and $g_{k} \, (k\ge 1)$.
We set
\begin{align*}
        e_{k}=g_{k}+\hbar g_{k-1}
\end{align*}
for $k\ge 2$.
The set
$\mathcal{A}=\{e_{1}-g_{1}\} \sqcup \{g_{k} \mid k\ge 1\}$
is algebraically independent and generates $\widehat{\mathfrak{H}^{1}}$.
We denote by $\mathfrak{z}$ the free $\mathcal{C}$-module
spanned by $\mathcal{A}$.
We denote by $\widehat{\mathfrak{H}^{0}}$
the $\mathcal{C}$-submodule of $\widehat{\mathfrak{H}^{1}}$ spanned by
$1$ and the monomials of the form $u_{1}\cdots u_{r}$
$(u_{1}, \ldots , u_{r} \in \mathcal{A})$ with $u_{r}\not=e_{1}-g_{1}$.

Hereafter we assume that $0<q<1$.
We endow
$\mathbb{R}$ with the structure of a $\mathcal{C}$-module
such that $\hbar^{\pm 1}$ acts as multiplication by $(1-q)^{\pm 1}$.
For $m\ge 1$, we define the $\mathcal{C}$-linear map
$I_{q}(\, \cdot \, |\, m): \mathfrak{z}\to \mathbb{R}$ by
\begin{align*}
        I_{q}(e_{1}-g_{1}\,|\, m)=1-q, \qquad
        I_{q}(g_{k}\,|\, m)=\left(\frac{q^{m}}{[m]}\right)^{k}
        \quad (k\ge 1),
\end{align*}
where $[m]=(1-q^{m})/(1-q)$ is the $q$-integer.
We have
\begin{align*}
        I_{q}(e_{k}\,|\,m)=\frac{q^{(k-1)m}}{[m]^{k}}
        \qquad (k\ge 1).
\end{align*}

We define the $\mathcal{C}$-linear map
$Z_{q}: \widehat{\mathfrak{H}^{0}} \longrightarrow \mathbb{R}$ by
$Z_{q}(1)=1$ and
\begin{align*}
        Z_{q}(u_{1}\cdots u_{r})=
        \sum_{0<m_{1}<\cdots <m_{r}}
        \prod_{a=1}^{r}I_{q}(u_{a} \,|\, m_{a})
\end{align*}
for $u_{1}, \ldots , u_{r} \in \mathcal{A}$ with
$u_{r}\not=e_{1}-g_{1}$.
Then the Bradley-Zhao model \eqref{eq:BZ-model} is expressed as
\begin{align*}
        \zeta_{q}(\bk)=Z_{q}(e_{k_{1}} \cdots e_{k_{r}})
\end{align*}
for a tuple of positive integers $\bk=(k_{1}, \ldots , k_{r})$ with $k_{r}\ge 2$.

We extend the map $Z_{q}$
to the formal power series ring
$Z_{q}: \widehat{\mathfrak{H}^{0}}[[X]] \longrightarrow \mathbb{R}[[X]]$ by
\begin{align*}
        Z_{q}(\sum_{n\ge 0}w_{n}X^{n})=\sum_{n\ge 0}Z_{q}(w_{n})X^n
        \qquad
        (w_{n} \in \widehat{\mathfrak{H}^{0}}).
\end{align*}

%%%%%%%%%%%%%%

\subsection{Harmonic product}\label{sec:harmonic-product}

We define the symmetric $\mathcal{C}$-bilinear map
$\circ_{\hbar}:  \mathfrak{z} \times \mathfrak{z} \to
        \mathfrak{z}$ by
\begin{align*}
        (e_{1}-g_{1})\circ_{\hbar}(e_{1}-g_{1})=\hbar (e_{1}-g_{1}), \quad
        (e_{1}-g_{1})\circ_{\hbar} g_{k}=\hbar g_{k}, \quad
        g_{k}\circ_{\hbar} g_{l}=g_{k+l} \qquad (k, l\ge 1).
\end{align*}
It is associative,
that is,
$(u_{1} \circ_{\hbar} u_{2})\circ_{\hbar} u_{3}=u_{1}\circ_{\hbar}
        (u_{2} \circ_{\hbar} u_{3})$ for any $u_{1}, u_{2}, u_{3} \in \mathfrak{z}$.
For $u \in \mathfrak{z}$ and $n\ge 1$,
we define $u^{\circ_{\hbar}n}$ inductively by $u^{\circ_{\hbar}1}=u$ and
$u^{\circ_{\hbar}n}=(u^{\circ_{\hbar}(n-1)})\circ_{\hbar}u$ for $n\ge 2$.

The \textit{harmonic product} $\ast_{\hbar}$ on $\widehat{\mathfrak{H}^{1}}$
is the $\mathcal{C}$-bilinear binary operation
uniquely defined by the following properties:
\begin{enumerate}
        \item For any $w \in \widehat{\mathfrak{H}^{1}}$,
              it holds that $w \ast_{\hbar} 1=w$ and $1\ast_{\hbar} w=w$.
        \item For any $w, w'\in \widehat{\mathfrak{H}^{1}}$ and $u, v \in \mathcal{A}$,
              it holds that $(w u)\ast_{\hbar}(w' v)=(w\ast_{\hbar} w' v)u+(wu\ast_{\hbar} w')v+
                      (w\ast_{\hbar} w')(u\circ_{\hbar}v)$.
\end{enumerate}
The harmonic product is commutative and associative, and
$\widehat{\mathfrak{H}^{0}}$ is closed under it
since $g_{k}\circ_{\hbar} g_{l}=g_{k+l}$ for $k, l \ge 1$.
For $w \in \widehat{\mathfrak{H}^{1}}$
and $n\ge 1$,
we define $w^{\ast_{\hbar} n}$ inductively by
$w^{\ast_{\hbar} 1}=w$ and
$w^{\ast_{\hbar} n}=(w^{\ast_{\hbar}(n-1)})\ast_{\hbar}w$
for $n\ge 2$.

Since
\begin{align}
        I_{q}(u\circ_{\hbar} v\,|\, m)=I_{q}(u \,|\, m)I_{q}(v \,|\, m)
        \qquad
        (u, v \in \mathfrak{z}, \, m\ge 1),
        \label{eq:I-hom}
\end{align}
the map $Z_{q}$ is a homomorphism with
respect to the harmonic product.
Namely, it holds that
\begin{align}
        Z_{q}(w \ast_{\hbar} w')=Z_{q}(w) Z_{q}(w')
        \label{eq:Zq-harmonic}
\end{align}
for any $w, w' \in \widehat{\mathfrak{H}^{0}}$.

We extend the harmonic product to
the formal power series ring $\widehat{\mathfrak{H}^{1}}[[X]]$
by
$f(X)\ast_{\hbar}g(X)=\sum_{n \ge 0}(\sum_{k+l=n}a_{k}\ast_{\hbar}b_{l})X^{n}$
for $f(X)=\sum_{k\ge 0}a_{k}X^{k}$ and $g(X)=\sum_{l \ge 0}b_{l}X^{l}$.
We also extend the map $\circ_{\hbar}$ to
the $\mathcal{C}$-module
$\mathfrak{z}[[X]]=\{\sum_{n\ge 0}a_{n}X^{n}\mid a_{n} \in \mathfrak{z} \, (n\ge 0)\}$
by $f(X)\circ_{\hbar}g(X)=\sum_{n \ge 0}(\sum_{k+l=n}a_{k}\circ_{\hbar}b_{l})X^{n}$.

For $f(X) \in X \widehat{\mathfrak{H}^{1}}[[X]]$, we define
the exponential by
\begin{align*}
        \exp_{\ast_{\hbar}}(f(X))=1+\sum_{n=1}^{\infty}\frac{1}{n!}
        (f(X))^{\ast_{\hbar} n}
\end{align*}
which is well-defined as an element of
$\widehat{\mathfrak{H}^{1}}[[X]]$ because
$f(X)$ does not have a constant term.
Then we have
\begin{align*}
        Z_{q}(\exp_{\ast_{\hbar}}{(f(X))})=\exp{(Z_{q}(f(X)))}
\end{align*}
for any $f(X)\in X\widehat{\mathfrak{H}^{0}}[[X]]$.
%%%%%%%%%%%%%%%%%%%%%%%

\subsection{Bachmann's identity}

We rewrite Bachmann's identity \eqref{eq:Bachmann} in the algebraic framework.
First we see that
\begin{align*}
        A_{r}(q)=(1-q)^{-2r}\sum_{0<m_{1}<\cdots <m_{r}}
        \prod_{a=1}^{r}I_{q}(e_{2}\mid m_{a})=Z_{q}(\hbar^{-2r}e_{2}^{r}).
\end{align*}
Next the function $G_{k}(q)$ is realized as follows.

\begin{prop}\label{prop:phik-G}
        For $k\ge 1$, we set
        \begin{align}
                \phi_{k}=\sum_{j=1}^{k}(j-1)!
                {k\brace j}
                \hbar^{k-j}g_{j}.
                \label{eq:def-phi}
        \end{align}
        Then we have $Z_{q}(\hbar^{-k}\phi_{k})=G_{k}(q)$.
\end{prop}

\begin{proof}
        We consider the generating function
        \begin{align*}
                \Phi(T)=\sum_{k\ge 1}\frac{T^{k-1}}{(k-1)!}Z_{q}(\hbar^{-k}\phi_{k})=
                \sum_{k\ge 1}T^{k-1}
                \sum_{j=1}^{k}\frac{(j-1)!}{(k-1)!}{k\brace j}
                \sum_{m>0}
                \left(\frac{q^{m}}{1-q^{m}}\right)^{j}.
        \end{align*}
        By using \eqref{eq:stirling-second-exp},
        we see that
        \begin{align*}
                \Phi(T) & =
                \sum_{j\ge 1}\sum_{m>0}
                (j-1)! \left(\frac{q^{m}}{1-q^{m}}\right)^{j}
                \frac{d}{dT}\frac{(e^{T}-1)^{j}}{j!}                            \\
                        & =-\sum_{m>0}
                \frac{d}{dT}\log{\left(1-\frac{q^{m}}{1-q^{m}}(e^{T}-1)\right)} \\
                        & =
                \sum_{m>0} \frac{q^{m}e^{T}}{1-q^{m}e^{T}}
                =\sum_{m>0} \frac{e^{T+y}}{1-e^{T+y}}
                \bigg|_{y=m\log{q}}.
        \end{align*}
        Since
        \begin{align*}
                \frac{e^{T+y}}{1-e^{T+y}}=
                e^{T\frac{d}{dy}}\frac{e^{y}}{1-e^{y}}=
                \sum_{k\ge 1}\frac{T^{k-1}}{(k-1)!}\left(\frac{d}{dy}\right)^{k-1}
                \sum_{n\ge 1}e^{ny}=
                \sum_{k \ge 1}\frac{T^{k-1}}{(k-1)!}\sum_{n \ge 1}n^{k-1}e^{ny},
        \end{align*}
        we find that
        \begin{align*}
                \Phi(T)=\sum_{m>0}
                \sum_{k\ge 1}\frac{T^{k-1}}{(k-1)!}
                \sum_{n \ge 1}n^{k-1}q^{mn}=
                \sum_{k\ge 1}\frac{T^{k-1}}{(k-1)!}G_{k}(q).
        \end{align*}
\end{proof}

Proposition \ref{prop:phik-G} implies that
the identity \eqref{eq:Bachmann} is written as
\begin{align*}
        Z_{q}\left(\frac{1}{1-\hbar^{-2}e_{2}X^2}\right)=
        \exp{\left(2
        \sum_{k\ge 1}\frac{(-1)^{k-1}}{(2k)!}Z_{q}(\hbar^{-2k}\phi_{2k})
        (2\arcsin{(X/2)})^{2k}
        \right)}.
\end{align*}
Since the map $Z_{q}$ is a homomorphism with respect to the harmonic product,
the above equality follows from
\begin{align}
        \frac{1}{1-\hbar^{-2}e_{2}X^2}=\exp_{\ast_{\hbar}}{\left(
                2\sum_{k\ge 1}\frac{(-1)^{k-1}}{(2k)!} \hbar^{-2k}\phi_{2k}
                (2\arcsin{(X/2)})^{2k}
                \right)},
        \label{eq:Bachmann3}
\end{align}
which is an identity of formal power series
with coefficients in $\widehat{\mathfrak{H}^{0}}$.
In the next section we prove a general identity which contains \eqref{eq:Bachmann3}
as a special case.

\begin{rem}\label{rem:KMS}
        Let $N$ be a positive integer and $S$ a non-empty subset of $\mathbb{Z}/N\mathbb{Z}$.
        For $m \in \mathbb{Z}$, we write $m \in S$ if $(m \, \mathrm{mod}\, N)\in S$.
        For $\epsilon \in \{ +1, -1\}$ and $r\ge 1$, set
        \begin{align*}
                A_{S, N, \epsilon, r}(q)=
                \sum_{\substack{0<m_{1}<\cdots <m_{r}
                \\ m_{1}, \ldots , m_{r} \in S}}
                \frac{\epsilon^{r}q^{m_{1}+\cdots +m_{r}}}
                {(1-\epsilon q^{m_{1}})^{2}\cdots (1-\epsilon q^{m_{r}})^{2}}.
        \end{align*}
        Kang, Matsusaka and Shin
        \cite{KMS} proves a generalization of \eqref{eq:Bachmann} which states that
        \begin{align}
                1+\sum_{r\ge 1}A_{S, N, \epsilon, r}(q)X^{2r}=
                \exp{\left(2\sum_{k\ge 1}\frac{(-1)^{k-1}}{(2k)!}G_{S, N, \epsilon, 2k}(q)
                (2\arcsin{(X/2)})^{2k}\right)},
                \label{eq:KMS}
        \end{align}
        where $G_{S, N, \epsilon, k}(q)$ is defined by
        \begin{align*}
                G_{S, N, \epsilon, k}(q)
                =\sum_{n\ge 1}\left(
                \sum_{\substack{d\mid n
                \\ n/d \in S}}\epsilon^{d}d^{k-1} \right)
                q^{n}
        \end{align*}
        for $k\ge 1$.
        By setting $N=1$ and $\epsilon=+1$, we recover \eqref{eq:Bachmann}.

        The identity \eqref{eq:KMS} can be derived from \eqref{eq:Bachmann3}
        as follows.
        For $\epsilon \in \{ +1, -1\}$ and $m\ge 1$, set
        \begin{align*}
                I_{q, \epsilon}(e_{1}-g_{1}\,|\, m)=1-q, \qquad
                I_{q, \epsilon}(g_{k}\,|\, m)=
                \left((1-q)\frac{\epsilon q^{m}}{1-\epsilon q^{m}}\right)^{k}
                \quad (k\ge 1).
        \end{align*}
        We define the $\mathcal{C}$-linear map
        $Z_{q, S, N, \epsilon}: \widehat{\mathfrak{H}^{0}} \longrightarrow \mathbb{R}$ by
        $Z_{q, S, N, \epsilon}(1)=1$ and
        \begin{align*}
                Z_{q, S, N, \epsilon}(u_{1}\cdots u_{r})=
                \sum_{\substack{0<m_{1}<\cdots <m_{r}
                \\ m_{1}, \ldots, m_{r} \in S}}
                \prod_{a=1}^{r}I_{q, \epsilon}(u_{a} \,|\, m_{a})
        \end{align*}
        for $u_{1}, \ldots , u_{r}\in \mathcal{A}$ with $u_{r}\not=e_{1}-g_{1}$.
        Then we have
        \begin{align*}
                Z_{q, S, N, \epsilon}(\hbar^{-2r}e_{2}^{r})=
                A_{S, N, \epsilon, r}(q), \quad
                Z_{q, S, N, \epsilon}(\phi_{k})=G_{S, N, \epsilon, k}(q).
        \end{align*}

        It can be checked that
        the relation \eqref{eq:I-hom} with $I_{q}$ replaced by $I_{q, \epsilon}$
        also holds.
        Hence the map
        $Z_{q, S, N, \epsilon}$ is a homomorphism with respect to
        the harmonic product,
        and the identity \eqref{eq:KMS} follows from
        \eqref{eq:Bachmann3}.
\end{rem}

%%%%%%%%%%%%%%%%%%%%%%%%%%%%%%

\section{General identity}\label{sec:main}

In this section we prove a general identity
which contains \eqref{eq:Bachmann3} as a special case.
Hereafter we set
\begin{align*}
        \epsilon_{N}=\exp{(2\pi i /N)} \qquad (N \ge 1).
\end{align*}

We set
\begin{align*}
        \mathcal{P}=\left\{P(z) \in \mathbb{Q}[z] \mid P(0)=P(1)=0 \right\}
\end{align*}
and define the $\mathbb{Q}$-linear map
$\psi: \mathcal{P} \to \mathfrak{z}$ by
$\psi(z^{n})=(-\hbar)^{-n}g_{n}$ for $n\ge 1$.
For example, we have
\begin{align*}
        \psi(z^{N-1}(z-1))=(-\hbar)^{-N}g_{N}-(-\hbar)^{-(N-1)}g_{N-1}=(-\hbar)^{-N}e_{N}
        \qquad (N\ge 2).
\end{align*}
The image of the map $\psi$ coincides with the $\mathbb{Q}$-module
spanned by the set $\{\hbar^{-k}e_{k}\}_{k\ge 2}$.
For simplicity of notation, we write $\psi_{P}$ instead of $\psi(P(z))$.

The general identity is given as follows.

\begin{thm}\label{thm:main1}
        Suppose that $P(z)\in \mathcal{P}\setminus\{0\}$ and $L\ge 1$.
        Set $N=\deg{P}\ge 2$.
        Then it holds that
        \begin{align}
                \frac{1}{1+(\psi_{P})^{\circ_{\hbar}L}X^{NL}}=
                \exp_{\ast_{\hbar}}{\left(-\sum_{k\ge 2}\frac{\hbar^{-k}}{k!}\phi_{k}
                        \sum_{m=1}^{NL}\left(\mathop{\mathrm{Log}}{(1-\alpha(\epsilon_{NL}^{m}X))}\right)^{k}
                        \right)},
                \label{eq:main1}
        \end{align}
        where
        $\alpha(X)$ is a convergent power series satisfying
        $\alpha(0)=0$ and
        \begin{align}
                1-x^{N}P(z)=\prod_{m=1}^{N}(1-\alpha(\epsilon_{N}^{m}x)z)
                \label{eq:def-alpha}
        \end{align}
        in a neighborhood of $x=0$.
\end{thm}

\begin{rem}\label{rem:existence-alpha}
        We see the existence of such a power series $\alpha(X)$ as follows.
        Set $R(z)=z^{N}P(1/z)$. Since $R(0)\not=0$,
        we can take a holomorphic branch of $R(z)^{1/N}$
        in a neighborhood of $z=0$.
        Then the function $f(x, z)=z-x\,R(z)^{1/N}$
        is holomorphic in a neighborhood of $(x, z)=(0, 0)$
        and $\partial f/\partial z$ does not vanish at $(0, 0)$.
        Hence there exists a unique holomorphic function $z=\alpha(x)$ satisfying
        $\alpha(0)=0$ and $f(x, \alpha(x))=0$ in a neighborhood of $x=0$.
                {}From the definition of $f$, it holds that $\alpha(x)^{N}=x^{N}R(\alpha(x))$,
        and hence $z=1/\alpha(x)$ is a root of the polynomial $1-x^{N}P(z)$.
        Although the function $\alpha(x)$ depends on the choice of the branch of $R(z)^{1/N}$,
        the set $\{\alpha(\epsilon_{N}^{m}x)\}_{m=1}^{N}$ does not
        and its elements are distinct around $x=0$.
        Therefore we have the factorization \eqref{eq:def-alpha}.
\end{rem}

\begin{rem}
        The logarithms part of the right-hand side of \eqref{eq:main1} can be
        calculated algebraically from $P(z)$ in principle as follows.
        It follows from \eqref{eq:stirling-first-log} that
        \begin{align}
                \sum_{m=1}^{NL}(\mathop{\mathrm{Log}}(1-\alpha(\epsilon_{NL}^{m}X)))^{k}=(-1)^{k}
                \sum_{s \ge k}\frac{k!}{s!}
                {s \brack k} %\left[ s \atop k \right] 
                \sum_{m=1}^{NL}(\alpha(\epsilon_{NL}^{m}X))^{s}.
                \label{eq:rem-log-sum}
        \end{align}
        On the other hand, from \eqref{eq:def-alpha}, we have
        \begin{align*}
                1-X^{NL}(P(z))^{L}=\prod_{a=1}^{L}(1-\epsilon_{L}^{a}X^{N}P(z))=
                \prod_{m=1}^{NL}(1-\alpha(\epsilon_{NL}^{m}X)z).
        \end{align*}
        By taking the logarithmic derivative with respect to $z$, we see that
        \begin{align*}
                \frac{LX^{NL}P'(z)(P(z))^{L-1}}{1-X^{NL}(P(z))^{L}}=\sum_{m=1}^{NL}
                \frac{\alpha(\epsilon_{NL}^{m}X)}{1-\alpha(\epsilon_{NL}^{m}X)z}=
                \sum_{s\ge 1}z^{s-1}\sum_{m=1}^{NL}(\alpha(\epsilon_{NL}^{m}X))^{s}.
        \end{align*}
        Hence we can write the power sum
        $\sum_{m=1}^{NL}(\alpha(\epsilon_{NL}^{m}X))^{s}$ for $s\ge 1$
        as a power series in $X^{NL}$ with rational coefficients
        by expanding the left-hand side at $z=0$.
        Thus we obtain a formula for the logarithmic part \eqref{eq:rem-log-sum}.
        Note that the logarithmic part becomes zero when $k=1$
        because $\prod_{m=1}^{NL}(1-\alpha(\epsilon_{NL}^{m}X))=1$
        under the assumption that $P(1)=0$.
\end{rem}

Now we start to prove Theorem \ref{thm:main1}.
We first show that the desired equality \eqref{eq:main1} with $L\ge 2$
follows from that with $L=1$.

\begin{prop}\label{prop:frac-prod-ast}
        For any $f(X), g(X) \in \mathfrak{z}[[X]]$, it holds that
        \begin{align}
                \frac{1}{1+Xf(X)}\ast_{\hbar}\frac{1}{1+Xg(X)}=
                \frac{1}{1+\left\{(f(X)+g(X))X-(f(X)\circ_{\hbar}g(X))X^{2}\right\}}
                \label{eq:frac-prod-ast}
        \end{align}
        in the $\mathcal{C}$-module $\mathfrak{z}[[X]]$.
\end{prop}

\begin{proof}
        Denote by $I$ the left-hand side.
        We see
        \begin{align*}
                I & =\left(1-\frac{1}{1+Xf(X)}Xf(X)\right)\ast_{\hbar}
                \left(1-\frac{1}{1+Xg(X)}Xg(X)\right)                  \\
                  & =1-\frac{Xf(X)}{1+Xf(X)}-\frac{Xg(X)}{1+Xg(X)}+
                \frac{X}{1+Xf(X)}f(X)\ast_{\hbar} \frac{X}{1+Xg(X)}g(X)
        \end{align*}
        and
        \begin{align*}
                 &
                \frac{X}{1+Xf(X)}f(X)\ast_{\hbar} \frac{X}{1+Xg(X)}g(X)                  \\
                 & =\left(\frac{X}{1+Xf(X)}\ast_{\hbar}\frac{Xg(X)}{1+Xg(X)}\right)f(X)+
                \left(\frac{Xf(X)}{1+Xf(X)}\ast_{\hbar}\frac{X}{1+Xg(X)}\right)g(X)      \\
                 & \quad {}+
                X^{2}I(f(X)\circ_{\hbar}g(X))                                            \\
                 & =X\left(\frac{1}{1+Xf(X)}-I\right)f(X)+
                X\left(\frac{1}{1+Xg(X)}-I\right)g(X)+
                X^{2}I(f(X)\circ_{\hbar}g(X)).
        \end{align*}
        Thus we find
        \begin{align*}
                I=1-I\left\{(f(X)+g(X))X-(f(X)\circ_{\hbar}g(X))X^{2}\right\}.
        \end{align*}
        This completes the proof.
\end{proof}

Set $\mathcal{C}_{\mathbb{C}}=\mathbb{C}[\hbar, \hbar^{-1}]=
        \mathbb{C}\otimes_{\mathbb{Q}}\mathcal{C}$.
We extend the coefficient ring $\mathcal{C}$ of
$\mathfrak{z}$ and $\widehat{\mathfrak{H}^{1}}$
to $\mathcal{C}_{\mathbb{C}}$.
Then the equality \eqref{eq:frac-prod-ast} still holds for any
power series $f(X)$ and $g(X)$ with coefficients in
$\mathfrak{z}_{\mathbb{C}}=\mathcal{C}_{\mathbb{C}}\otimes_{\mathcal{C}}\mathfrak{z}$.

\begin{cor}\label{cor:cyclotomic-prod}
        For any $u \in \mathfrak{z}_{\mathbb{C}}$ and $L\ge 1$, it holds that
        \begin{align*}
                \frac{1}{1+uX} \ast_{\hbar} \frac{1}{1+\epsilon_{L}uX} \ast_{\hbar}
                \cdots \ast_{\hbar} \frac{1}{1+\epsilon_{L}^{L-1}uX}=
                \frac{1}{1+u^{\circ_{\hbar} L}X^{L}}.
        \end{align*}
\end{cor}

\begin{proof}
        By using Proposition \ref{prop:frac-prod-ast} repeatedly,
        we see that the left-hand side is equal to
        \begin{align*}
                \left(1+\sum_{a=1}^{L}(-1)^{a-1}
                e_{a}(1, \epsilon_{L}, \ldots, \epsilon_{L}^{L-1})
                u^{\circ_{\hbar} a}X^{a} \right)^{-1},
        \end{align*}
        where
        $e_{a}(t_{1}, \ldots , t_{L})=\sum_{1\le j_{1}<\cdots <j_{a}\le L}t_{j_{1}}\cdots t_{j_{a}}$
        is the elementary symmetric polynomial.
        The desired equality follows from
        $        e_{a}(1, \epsilon_{L}, \ldots , \epsilon_{L}^{L-1})=(-1)^{L-1}\delta_{a, L}$
        for $1\le a \le L$.
\end{proof}

Now suppose that \eqref{eq:main1} holds for $L=1$.
By using Corollary \ref{cor:cyclotomic-prod}, we see that
\begin{align*}
         &
        \frac{1}{1+(\psi_{P})^{\circ_{\hbar}L}X^{NL}}  \\
         & =\frac{1}{1+\psi_{P}X^{N}}\ast_{\hbar}
        \frac{1}{1+\psi_{P}(\epsilon_{NL}X)^{N}}\ast_{\hbar}
        \frac{1}{1+\psi_{P}(\epsilon_{NL}^{2}X)^{N}}\ast_{\hbar} \cdots
        \ast_{\hbar}
        \frac{1}{1+\psi_{P}(\epsilon_{NL}^{L-1}X)^{N}} \\
         & =
        \exp_{\ast_{\hbar}}{\left(
                -\sum_{a=0}^{L-1}\sum_{k\ge 2}\frac{\hbar^{-k}}{k!}\phi_{k}
                \sum_{m=1}^{N}\left(\mathop{\mathrm{Log}}{(1-\alpha(\epsilon_{N}^{m}\epsilon_{NL}^{a}X))}\right)^{k}
                \right)}.
\end{align*}
It is equal to the right-hand side of \eqref{eq:main1} because
$\{\epsilon_{N}^{m}\epsilon_{NL}^{a}\mid 1\le m \le N, \, 0\le a\le L-1\}=
        \{\epsilon_{NL}^{m}\mid 1\le m \le NL\}$.
Thus it suffices to show the desired identity \eqref{eq:main1}
in the case of $L=1$.

To prove the identity \eqref{eq:main1} with $L=1$,
we use the following formula.

\begin{prop}\label{prop:exp-ast}
        \cite[Corollary 5.1]{HI}\,
        For any $u \in \mathfrak{z}$, it holds that
        \begin{align*}
                \frac{1}{1-uX}=
                \exp_{\ast_{\hbar}}{\left(
                        \sum_{n\ge 1}\frac{(-1)^{n-1}}{n}u^{\circ_{\hbar}n}X^{n}
                        \right)}
        \end{align*}
        in the formal power series ring $\widehat{\mathfrak{H}^{1}}[[X]]$.
\end{prop}

Set
\begin{align*}
        \Psi(X)=\sum_{r\ge 1}        \frac{X^{Nr}}{r}
        \left(\psi_{P}\right)^{\circ_{\hbar}r}.
\end{align*}
Proposition \ref{prop:exp-ast} implies that
\begin{align*}
        \frac{1}{1+\psi_{P}X^{N}}=\exp_{\ast_{\hbar}}(-\Psi(X)).
\end{align*}

Set $P(z)=\sum_{j=1}^{N}c_{j}z^{j}$ with $c_{j} \in \mathbb{Q}$.
Then $\psi_{P}=\sum_{j=1}^{N}c_{j}(-\hbar)^{-j}g_{j}$
and
\begin{align*}
        \Psi(X)
         & =\sum_{r\ge 1}\frac{X^{Nr}}{r}
        \sum_{j_{1}, \ldots , j_{r}=1}^{N}(-\hbar)^{-\sum_{a=1}^{r}j_{a}}
        c_{j_{1}} \cdots c_{j_{r}} g_{j_{1}+\cdots +j_{r}} \\
         & =\sum_{l \ge 1}(-\hbar)^{-l}g_{l}
        \sum_{r\ge 1}\frac{X^{Nr}}{r}
        \sum_{\substack{j_{1}+\cdots +j_{r}=l              \\ 1\le j_{1}, \ldots , j_{r}\le N}}
        c_{j_{1}} \cdots c_{j_{r}}.
\end{align*}
By using \eqref{eq:stirling-duality},
we see that
\begin{align*}
        g_{k}=\frac{1}{(k-1)!}\sum_{j=1}^{k}{k\brack j}
        (-\hbar)^{k-j}\phi_{j}
        \qquad (k\ge 1)
\end{align*}
{}from the definition \eqref{eq:def-phi} of $\phi_{k}$.
Hence
\begin{align*}
        \Psi(X)=\sum_{k\ge 1}(-\hbar)^{-k}\phi_{k}
        \sum_{r\ge 1}\frac{X^{Nr}}{r}
        \sum_{l\ge 1}
        \frac{1}{(l-1)!}{l \brack k}
        \sum_{\substack{j_{1}+\cdots +j_{r}=l \\ 1\le j_{1}, \ldots , j_{r}\le N}}
        c_{j_{1}} \cdots c_{j_{r}}.
\end{align*}
It follows from \eqref{eq:stirling-first-harmonic} and
the assumption $P \in \mathcal{P}$ that
the term with $k=1$ of the right-hand side vanishes because
\begin{align*}
        \sum_{l\ge 1}\frac{1}{(l-1)!}
        {l \brack 1}
        \sum_{\substack{j_{1}+\cdots +j_{r}=l \\ 1\le j_{1}, \ldots , j_{r}\le N}}
        c_{j_{1}} \cdots c_{j_{r}}=
        \left(\sum_{1\le j \le N}c_{j}\right)^{r}=(P(1))^{r}=0.
\end{align*}
For the terms with $k\ge 2$, we use the following formula.
\begin{lem}\label{lem:stirring-integral}
        For $k \ge 2$ and $l\ge 1$, it holds that
        \begin{align*}
                \frac{1}{(l-1)!}{l \brack k}=
                \frac{1}{(k-1)!}
                \left(\prod_{a=1}^{k-1}\int_{0}^{1}dz_{a}\right)
                G_{l}(z_{1}, \ldots , z_{k-1}),
        \end{align*}
        where $G_{l}(z_{1}, \ldots z_{k-1})$ is given by
        \begin{align*}
                G_{l}(z_{1}, \ldots , z_{k-1})=
                \prod_{j=1}^{k-1}\frac{1}{1-z_{j}}+
                \sum_{a=1}^{k-1}\frac{z_{a}^{l-1}}{z_{a}-1}
                \prod_{\substack{j=1 \\ j\not=a}}^{k-1}\frac{1}{z_{a}-z_{j}}.
        \end{align*}
\end{lem}

\begin{proof}
        By using \eqref{eq:stirling-first-harmonic},
        we see that
        \begin{align*}
                \frac{1}{(l-1)!}{l \brack k}=
                \sum_{l>m_{1}>\cdots >m_{k-1}>0}\frac{1}{m_{1}\cdots m_{k-1}}=
                \left(\prod_{a=1}^{k-1}\int_{0}^{1}dy_{a}\right)
                \sum_{l>m_{1}>\cdots >m_{k-1}>0} \prod_{a=1}^{k-1}y_{a}^{m_{a}-1}.
        \end{align*}
        It holds that
        \begin{align*}
                \sum_{l>m_{1}>\cdots >m_{k-1}>0} \prod_{a=1}^{k-1}y_{a}^{m_{a}-1}=
                \sum_{a=0}^{k-1}(-1)^{a}
                \frac{(y_{1}\cdots y_{a})^{l-1}}{y_{a+1}\cdots y_{k-1}}
                \prod_{j=1}^{a}\frac{1}{1-y_{j}\cdots y_{a}}
                \prod_{j=a+1}^{k-1}\frac{y_{a+1}\cdots y_{j}}{1-y_{a+1}\cdots y_{j}}
        \end{align*}
        (see, e.g., \cite[Proposition 3.3]{INY}).
        By changing the variables $y_{a}$ to $z_{a}=y_{1}\cdots y_{a} \, (1\le a<k)$,
        we obtain the desired formula because
        $G_{l}(z_{1}, \ldots , z_{k-1})$ is symmetric with respect to
        $z_{1}, \ldots , z_{k-1}$.
\end{proof}

Thus we find that
\begin{align*}
        \Psi(X)=\sum_{k\ge 1}\frac{(-\hbar)^{-k}}{(k-1)!}\phi_{k}
        \sum_{r\ge 1}\frac{X^{Nr}}{r}
        \sum_{l\ge 1}
        \left(\prod_{a=1}^{k-1}\int_{0}^{1}dz_{a}\right)
        G_{l}(z_{1}, \ldots , z_{k-1})
        \sum_{\substack{j_{1}+\cdots +j_{r}=l \\ 1\le j_{1}, \ldots , j_{r}\le N}}
        c_{j_{1}} \cdots c_{j_{r}}.
\end{align*}
Since
\begin{align*}
        \sum_{l\ge 1}G_{l}(z_{1}, \ldots , z_{k-1})
        \sum_{\substack{j_{1}+\cdots +j_{r}=l \\ 1\le j_{1}, \ldots , j_{r}\le N}}
        c_{j_{1}} \cdots c_{j_{r}}
        =(P(1))^{r}\prod_{j=1}^{k-1}\frac{1}{1-z_{j}}+
        \sum_{a=1}^{k-1}\frac{(P(z_{a}))^{r}}{z_{a}(z_{a}-1)}
        \prod_{\substack{j=1                  \\ j\not=a}}^{k-1}\frac{1}{z_{a}-z_{j}}
\end{align*}
and the first term is zero because $P \in \mathcal{P}$,
it holds that
\begin{align*}
        \Psi(X) & =
        \sum_{k\ge 2}\frac{(-\hbar)^{-k}}{(k-1)!}\phi_{k}
        \sum_{r\ge 1}\frac{X^{Nr}}{r}
        \left(\prod_{a=1}^{k-1}\int_{0}^{1}dz_{a}\right)
        \sum_{a=1}^{k-1}\frac{(P(z_{a}))^{r}}{z_{a}(z_{a}-1)}
        \prod_{\substack{j=1 \\ j\not=a}}^{k-1}\frac{1}{z_{a}-z_{j}} \\
                & =
        \sum_{k\ge 2}\frac{(-\hbar)^{-k}}{(k-1)!}\phi_{k}
        \left(\prod_{a=1}^{k-1}\int_{0}^{1}dz_{a}\right)
        \sum_{a=1}^{k-1}\frac{(-\mathop{\mathrm{Log}}{(1-X^{N}P(z_{a}))})}{z_{a}(z_{a}-1)}
        \prod_{\substack{j=1 \\ j\not=a}}^{k-1}\frac{1}{z_{a}-z_{j}}.
\end{align*}
Now we set
\begin{align*}
        F_{k}(x)=        \left(\prod_{a=1}^{k-1}\int_{0}^{1}dz_{a}\right)
        \sum_{a=1}^{k-1}\frac{(-\mathop{\mathrm{Log}}{(1-x^{N}P(z_{a}))})}{z_{a}(z_{a}-1)}
        \prod_{\substack{j=1 \\ j\not=a}}^{k-1}\frac{1}{z_{a}-z_{j}}
\end{align*}
for $k \ge 2$,
which is a holomorphic function defined in a neighborhood of $x=0$.
Then
\begin{align*}
        \Psi(X)=\sum_{k\ge 2}\frac{(-\hbar)^{-k}}{(k-1)!}\phi_{k}F_{k}(X).
\end{align*}
Hence, to verify \eqref{eq:main1} with $L=1$,
it suffices to show
\begin{align}
        F_{k}(x)=\frac{1}{k}\sum_{m=1}^{N}
        \left(-\mathop{\mathrm{Log}}(1-\alpha(\epsilon_{N}^{m}x))\right)^{k},
        \label{eq:F_k}
\end{align}
where $\alpha(x)$ is a convergent series satisfying \eqref{eq:def-alpha}.

Since $P \in \mathcal{P}$, we can set $P(z)=z(z-1)Q(z)$ with $Q(z)\in \mathbb{Q}[z]$.
Then
\begin{align*}
        F_{k}'(x)=Nx^{N-1}
        \left(\prod_{a=1}^{k-1}\int_{0}^{1}dz_{a}\right)
        \sum_{a=1}^{k-1}\frac{Q(z_{a})}{1-x^{N}P(z_{a})}
        \prod_{\substack{j=1 \\ j\not=a}}^{k-1}\frac{1}{z_{a}-z_{j}}.
\end{align*}
For simplicity of notation,
we set $\alpha_{m}=\alpha(\epsilon_{N}^{m}x)$
for $1\le m \le N$.
Then, since $\deg{Q}\le \deg{P}-2$ and
$1-x^{N}P(z)=\prod_{m=1}^{N}(1-\alpha_{m}z)$, it holds that
\begin{align*}
         &
        \sum_{a=1}^{k-1}\frac{Q(z_{a})}{1-x^{N}P(z_{a})}
        \prod_{\substack{j=1                                  \\ j\not=a}}^{k-1}\frac{1}{z_{a}-z_{j}}=
        \sum_{a=1}^{k-1}\mathrm{Res}_{w=z_{a}}dw
        \frac{Q(w)}{1-x^{N}P(w)}
        \prod_{j=1}^{k-1}\frac{1}{w-z_{j}}                    \\
         & =-\sum_{m=1}^{N}\mathrm{Res}_{w=\alpha_{m}^{-1}}dw
        \frac{Q(w)}{1-x^{N}P(w)}\prod_{j=1}^{k-1}\frac{1}{w-z_{j}}
        =x^{-N}\sum_{m=1}^{N}
        \frac{Q(\alpha_{m}^{-1})}{P'(\alpha_{m}^{-1})}
        \prod_{j=1}^{k-1}\frac{\alpha_{m}}{1-\alpha_{m}z_{j}}.
\end{align*}
We have $1-x^{N}P(\alpha_{m}^{-1})=0$
in a neighborhood of $x=0$.
By differentiating both sides with respect to $x$ and
using $P(z)=z(z-1)Q(z)$, we see that
\begin{align*}
        \frac{Q(\alpha_{m}^{-1})}{P'(\alpha_{m}^{-1})}=\frac{x}{N(1-\alpha_{m})}\,\frac{d\alpha_{m}}{dx}.
\end{align*}
Thus we find that
\begin{align*}
        F_{k}'(x)
         & =\sum_{m=1}^{N}\frac{1}{1-\alpha_{m}}\frac{d\alpha_{m}}{dx}
        \left(\prod_{a=1}^{k-1}\int_{0}^{1}dz_{a}\right)
        \prod_{j=1}^{k-1}\frac{\alpha_{m}}{1-\alpha_{m}z_{j}}                 \\
         & =\sum_{m=1}^{N}\frac{1}{1-\alpha_{m}}\frac{d\alpha_{m}}{dx}
        \left(\int_{0}^{1}dz \, \frac{\alpha_{m}}{1-\alpha_{m}z}\right)^{k-1} \\
         & =\sum_{m=1}^{N}\frac{1}{1-\alpha_{m}}\frac{d\alpha_{m}}{dx}
        \left(-\mathop{\mathrm{Log}}(1-\alpha_{m})\right)^{k-1}
        =\frac{d}{dx}\left(\frac{1}{k}\sum_{m=1}^{N}
        \left(-\mathop{\mathrm{Log}}(1-\alpha_{m})\right)^{k}\right).
\end{align*}
Since $F_{k}(0)=0$ and $\alpha_{m}|_{x=0}=0$ for $1\le m \le N$,
we obtain \eqref{eq:F_k}.
This completes the proof of \eqref{eq:main1} with $L=1$.

%%%%%%%%%%%%%%%%%%%%%%%%%%%%%%%%%%%%%%%%

\section{A solvable case}\label{sec:solvable}

In order to write the formula \eqref{eq:main1} explicitly,
we need to find a power series $\alpha(x)$ satisfying \eqref{eq:def-alpha}
and calculate the logarithm $\mathrm{Log}(1-\alpha(x))$.
As an example where such a calculation can be performed,
we consider the case where
\begin{align*}
        P(z)=z^{N-1}(z-1) \qquad (N\ge 2).
\end{align*}
Then we have $\psi_{P}=(-\hbar)^{-N}e_{N}$
as seen in the beginning of Section \ref{sec:main}.

\begin{prop}\label{prop:varphi}
        Define the power series $\varphi(\theta; x)$ in $x$ by
        \begin{align}
                \varphi(\theta; x)=\sum_{k\ge 1}\frac{x^{k}}{k!}\prod_{a=1}^{k-1}(k\theta-a).
                \label{eq:def-varphi}
        \end{align}
        Set
        \begin{align}
                \alpha(x)=1-e^{\varphi(1/N; -x)}.
                \label{eq:alpha-varphi}
        \end{align}
        Then, for $P(z)=z^{N-1}(z-1)$, we have
        \begin{align}
                1-x^{N}P(z)=\prod_{m=1}^{N}(1-\alpha(\epsilon_{N}^{m}x)z).
                \label{eq:alpha-factorize}
        \end{align}
\end{prop}

See Appendix \ref{sec:app2} for the proof of Proposition \ref{prop:varphi}.
It implies the following identity.

\begin{thm}\label{thm:main2}
        For $N\ge 2$ and $L\ge 1$, it holds that
        \begin{align}
                \frac{1}{1+\hbar^{-NL}e_{N}^{\circ_{\hbar L}}X^{NL}}=
                \exp_{\ast_{\hbar}}{\left(-\sum_{k\ge 2}\frac{\hbar^{-k}}{k!}\phi_{k}
                        \sum_{m=1}^{NL}\left(\varphi(1/N; \epsilon_{NL}^{m}X)\right)^{k}
                        \right)},
                \label{eq:main-eN}
        \end{align}
        where $\varphi(\theta; x)$ is defined by \eqref{eq:def-varphi}.
\end{thm}

\begin{proof}
        Substitute \eqref{eq:alpha-varphi} into \eqref{eq:main1}
        and change $X \to -X$.
\end{proof}

As seen below the series $\varphi(1/N; x)$ with $N=2$ is essentially equal to
the arcsin function.
Thus we obtain Bachmann's identity
within the algebraic framework.

\begin{cor}
        For $L\ge 1$, it holds that
        \begin{align}
                \frac{1}{1+(-1)^{L}\hbar^{-2L}e_{2}^{\circ_{\hbar}L}X^{2L}}=
                \exp_{\ast_{\hbar}}{\left(2
                        \sum_{k\ge 1}\frac{(-1)^{k-1}}{(2k)!}\hbar^{-2k}\phi_{2k}
                        \sum_{m=1}^{L}
                        \left( 2\arcsin{\frac{e^{m\pi i/L}X}{2}} \right)^{2k}
                        \right)}.
                \label{eq:main-e2}
        \end{align}
        In particular, the equality \eqref{eq:Bachmann3} holds.
\end{cor}

\begin{proof}
        We have
        \begin{align*}
                \varphi(\frac{1}{2}; x)=x\,
                {}_{2}F_{1}\left(\frac{1}{2}, \frac{1}{2}, \frac{3}{2}; -\frac{x^{4}}{2}\right)=
                -2i\arcsin{\frac{ix}{2}},
        \end{align*}
        where ${}_{2}F_{1}(\alpha, \beta, \gamma; z)$ is
        the Gauss hypergeometric series.
        Substitute it into \eqref{eq:main-eN} with $N=2$ and
        change $X$ to $iX$.
        Then the terms with odd $k$ vanish because
        the arcsin function is odd.
        Thus we obtain \eqref{eq:main-e2}.
\end{proof}

We write the sum over $1\le m \le NL$
in the right-hand side of $\eqref{eq:main-eN}$ more explicitly.

\begin{prop}\label{prop:varphi-power}
        For $k\ge 1$, it holds that
        \begin{align*}
                \frac{\left(\varphi(\theta; x)\right)^{k}}{k!}=
                \sum_{n\ge 1}\frac{x^{n}}{n!}
                \sum_{j\ge k}(-1)^{n-j}
                {n \brack j}
                \binom{j-1}{k-1}(n\theta)^{j-k}.
        \end{align*}
\end{prop}

See Appendix \ref{sec:app2} for the proof of
Proposition \ref{prop:varphi-power}.

For $n\ge 1$, we set $H_{0}(n)=1$ and
\begin{align*}
        H_{r}(n)=\sum_{0<m_{1}<\cdots <m_{r}<n}\frac{1}{m_{1}\cdots m_{r}}
\end{align*}
for $r\ge 1$.
Note that $H_{r}(n)=0$ if $1\le n\le r$.
For $m \ge 0$ and $n_{1}, n_{2} \ge 1$, we set
\begin{align*}
        C_{m}(n_{1}, n_{2})=\sum_{a=0}^{m}(-1)^{m-a}
        H_{a}(n_{1})H_{m-a}(n_{2}).
\end{align*}
We have $C_{m}(n_{1}, n_{2})=0$ if $m\ge n_{1}+n_{2}-1$.

\begin{thm}
        For $N\ge 2$ and $L \ge 1$, it holds that
        \begin{align}
                 &
                \frac{1}{1+(-1)^{(N-1)L}\hbar^{-NL}e_{N}^{\circ_{\hbar} L}X}
                \label{eq:main-eN-explicit}    \\
                 & =\exp_{\ast_{\hbar}}{\left(
                \frac{N}{(N-1)L}\sum_{n\ge 1}
                \frac{X^{n}}{n^{2}\binom{NLn}{Ln}}
                \sum_{k=2}^{NLn}C_{k-2}(Ln, (N-1)Ln)\hbar^{-k}\phi_{k}
                \right)}.
                \nonumber
        \end{align}
\end{thm}

\begin{proof}
        For $k\ge 1$, we define the polynomial $D_{k}(x)$ by
        \begin{align*}
                D_{k}(x)=\sum_{j\ge k}(-1)^{NLn-j}{NLn \brack j}
                \binom{j-1}{k-1}x^{j-k}.
        \end{align*}
        {}From Proposition \ref{prop:varphi-power} and
        \begin{align*}
                \sum_{m=1}^{NL}(\epsilon_{NL}^{m}x)^{n}=\left\{
                \begin{array}{ll}
                        NL \, x^{n} & (NL \mid  n)          \\
                        0           & (\textrm{otherwise}),
                \end{array}
                \right.
        \end{align*}
        we see that
        \begin{align*}
                \frac{1}{k!}\sum_{m=1}^{NL}(\varphi(1/N; \epsilon^{m}_{NL}x))^{k}=
                NL\sum_{n\ge 1}\frac{x^{NLn}}{(NLn)!}D_{k}(Ln).
        \end{align*}
        It follows from \eqref{eq:shifted-factorial} that
        \begin{align*}
                D_{k}(x)
                 & =\frac{1}{(k-1)!}\left(\frac{d}{dx}\right)^{k-1}
                \sum_{j\ge 1}(-1)^{NLn-j}{NLn \brack j}
                x^{j}                                               \\
                 & =\frac{1}{(k-1)!}\left(\frac{d}{dx}\right)^{k-1}
                \prod_{a=1}^{NLn-1}(x-a)                            \\
                 & =\sum_{\substack{p+q+r=k-1                       \\ p, q, r\ge 0}}
                \frac{1}{p!q!r!}
                \left\{\left(\frac{d}{dx}\right)^{p}\prod_{a=1}^{Ln-1}(x-a)\right\}
                \left\{\left(\frac{d}{dx}\right)^{q}(x-Ln)\right\}
                \left\{\left(\frac{d}{dx}\right)^{r}\prod_{a=Ln+1}^{NLn-1}(x-a)\right\}.
        \end{align*}
        We set $x=Ln$. Then only the terms with $q=1$ remain.
        Hence
        \begin{align*}
                D_{k}(Ln)
                 & =\sum_{\substack{p+r=k-2                                     \\ p, r\ge 0}}
                \frac{1}{p!r!}
                \left\{\left(\frac{d}{dx}\right)^{p}\prod_{a=1}^{Ln-1}(x-a)\right\}
                \left\{\left(\frac{d}{dx}\right)^{r}\prod_{a=Ln+1}^{NLn-1}(x-a)\right\}\bigg|_{x=Ln}
                \\
                 & =(Ln-1)!((N-1)Ln-1)!
                \sum_{\substack{p+r=k-2                                         \\ p, r\ge 0}}
                (-1)^{(N-1)Ln-1-r}
                H_{p}(Ln)H_{r}((N-1)Ln)
                \\
                 & =(-1)^{(N-1)Ln-1}\frac{(NLn)!}{(N-1)(Ln)^{2}\binom{NLn}{Ln}}
                C_{k-2}(Ln, (N-1)Ln).
        \end{align*}
        Therefore, the exponent in the right-hand side of \eqref{eq:main-eN}
        is equal to
        \begin{align*}
                 &
                {}-\sum_{k\ge 2}\frac{\hbar^{-k}}{k!}\phi_{k}
                \sum_{m=1}^{NL}\left(\varphi(1/N; \epsilon_{NL}^{m}X)\right)^{k}
                =-\sum_{k\ge 2}\hbar^{-k}\phi_{k}NL\sum_{n\ge 1}
                \frac{x^{NLn}}{(NLn)!}D_{k}(Ln) \\
                 & =\frac{N}{(N-1)L}
                \sum_{n\ge 1}\frac{((-1)^{(N-1)L}X^{NL})^{n}}{n^2 \binom{NLn}{Ln}}
                \sum_{k\ge 2}C_{k-2}(Ln, (N-1)Ln)\hbar^{-k}\phi_{k}.
        \end{align*}
        By changing $X^{NL} \to (-1)^{(N-1)L}X$, we obtain the desired identity.
\end{proof}

\begin{rem}
        If $N=2$, we have $C_{m}(Ln, (N-1)Ln)=C_{m}(Ln, Ln)$ for $m\ge 0$.
        It holds that, for $m\ge 0$ and $n\ge 1$,
        \begin{align*}
                C_{m}(n, n)=\sum_{a=0}^{m}(-1)^{m-a}H_{a}(n)H_{m-a}(n)=
                \left\{
                \begin{array}{ll}
                        (-1)^{m/2}H_{m/2}^{(2)}(n) & (\hbox{$m$: even}) \\
                        0                          & (\hbox{$m$: odd})
                \end{array}
                \right.
        \end{align*}
        where $H_{0}^{(2)}(n)=1$ and
        \begin{align*}
                H_{r}^{(2)}(n)=\sum_{0<m_{1}<\cdots <m_{r}<n}
                \frac{1}{m_{1}^{2}\cdots m_{r}^{2}}
        \end{align*}
        for $r\ge 1$.
        It is known that
        \begin{align}
                \left(2\arcsin{(X/2)}\right)^{2k}=(2k)!
                \sum_{n\ge k} \frac{H_{k-1}^{(2)}(n)}{n^2 \binom{2n}{n}}X^{2n}
                \label{eq:arcsin-expansion}
        \end{align}
        for $k\ge 1$ (see \cite{BC}).
        By changing $X\to X^{2L}$ in \eqref{eq:main-eN-explicit} and using
        \eqref{eq:arcsin-expansion},
        we recover \eqref{eq:main-e2}.
\end{rem}

Applying the map $Z_{q}$ to both sides of \eqref{eq:main-eN-explicit},
we obtain the following identity.

\begin{cor}
        Suppose that $N\ge 2$ and $L\ge 1$.
        We set
        \begin{align*}
                A_{r}^{(N, L)}(q)=
                \sum_{0<m_{1}<\cdots <m_{r}}
                \left(\frac{q^{(N-1)(m_{1}+\cdots +m_{r})}}
                {(1-q^{m_{1}})^{N} \cdots (1-q^{m_{r}})^{N}}
                \right)^{L}
        \end{align*}
        for $r\ge 1$.
        Then it holds that
        \begin{align*}
                 &
                1+\sum_{r\ge 1}(-1)^{((N-1)L-1)r}A_{r}^{(N, L)}(q)X^{r} \\
                 & =
                \exp{\left(
                \frac{N}{(N-1)L}\sum_{n\ge 1}
                \frac{X^{n}}{n^{2}\binom{NLn}{Ln}}
                \sum_{k=2}^{NLn}C_{k-2}(Ln, (N-1)Ln)G_{k}(q)
                \right)}.
        \end{align*}
\end{cor}

%%%%%%%%%%%%%%%%%%%%%%%%%%%%%%%%%%%

\section{Application to $\omega$MZVs}\label{sec:omega-MZV}

\subsection{Definition of $\omega$MZV}

We assume that $\omega >0$.
Let $M(\mathbb{C})$ be
the field of meromorphic functions on $\mathbb{C}$.
We endow $M(\mathbb{C})$ with the structure of a $\mathcal{C}$-module
such that $\hbar^{\pm 1}$ acts as multiplication by $(2\pi i \omega)^{\pm 1}$.
Then we define the $\mathcal{C}$-linear map
$J_{\omega}: \mathfrak{z} \to M(\mathbb{C})$ by
\begin{align*}
        J_{\omega}(e_{1}-g_{1} \, | \, t)=2\pi i \omega, \qquad
        J_{\omega}(g_{k} \, | \, t)=\left(
        \frac{2\pi i \omega \, e^{2\pi i \omega t}}{1-e^{2\pi i \omega t}}
        \right)^{k}
        =\left(        \frac{2\pi i \omega }{e^{-2\pi i \omega t}-1}
        \right)^{k}
\end{align*}
for $k \ge 1$.

For $u_{1}, \ldots , u_{r} \in \mathcal{A}$ such that $u_{r}\not=e_{1}-g_{1}$,
we consider the multiple integral
\begin{align}
        Z_{\omega}(u_{1}\cdots u_{r})=
        \prod_{a=1}^{r}\int_{{}-\epsilon+i\,\mathbb{R}}\frac{dt_{a}}{e^{2\pi i t_{a}}-1}
        \prod_{a=1}^{r}J_{\omega}(u_{a}\,|\, t_{1}+\cdots +t_{a}),
        \label{eq:def-omega-zeta}
\end{align}
where $\epsilon$ is a constant satisfying $0<\epsilon<\min\{1, 1/r \omega\}$.

\begin{prop}\label{prop:omega-zeta-convergence}
        If $\omega>0$, the integral \eqref{eq:def-omega-zeta} is absolutely convergent.
        It does not depend on $\epsilon$.
\end{prop}

\begin{rem}
        Although we assumed that $0<\omega<2$ for absolute convergence
        in the previous paper \cite{T-omega},
        the condition can be relaxed to $\omega>0$ as shown in the following proof.
\end{rem}

\begin{proof}
        {}From the assumption we can set $u_{r}=g_{k}$ with $k \ge 1$.
        Set $t_{a}=-\epsilon+iu_{a}$ with $u_{a} \in \mathbb{R}$ for  $1\le a \le r$.
        Then the integrand is estimated from above by
        $C\exp{(\pi S(u))}$, where $C$ is a positive constant which depends on
        $\epsilon, \omega, r$ and $k$, and $S(u)$ is given by
        \begin{align*}
                S(u)=\sum_{a=1}^{r}(u_{a}-|u_{a}|)-k\omega
                \left(\sum_{a=1}^{r}u_{a}+\left|\sum_{a=1}^{r}u_{a}\right|\right)
        \end{align*}
        (see the proof of \cite[Proposition 3.1]{T-omega}).
        Since $\omega>0$ and $k\ge 1$, we have
        \begin{align*}
                S(u) \le
                \sum_{a=1}^{r}(u_{a}-|u_{a}|)-\omega
                \left(\sum_{a=1}^{r}u_{a}+\left|\sum_{a=1}^{r}u_{a}\right|\right)
                =(1-\omega)\sum_{a=1}^{r}u_{a}-\sum_{a=1}^{r}|u_{a}|-
                \omega\left|\sum_{a=1}^{r}u_{a}\right|.
        \end{align*}
        Take a constant $\rho$ such that $-\omega<\rho<\min{\{\omega, 2-\omega\}}$.
        The right-hand side is equal to
        \begin{align*}
                \rho\sum_{a=1}^{r}u_{a}-\omega\left|\sum_{a=1}^{r}u_{a}\right|+
                \sum_{a=1}^{r}((1-\omega-\rho)u_{a}-|u_{a}|).
        \end{align*}
        It is estimated from above by
        \begin{align*}
                |\rho|\left|\sum_{a=1}^{r}u_{a}\right|-\omega\left|\sum_{a=1}^{r}u_{a}\right|+
                \sum_{a=1}^{r}(|1-\omega-\rho||u_{a}|-|u_{a}|).
        \end{align*}
        Since $|\rho|<\omega$, we obtain
        \begin{align*}
                S(u)\le -(1-|1-\omega-\rho|)\sum_{a=1}^{r}|u_{a}|.
        \end{align*}
        It holds that $|1-\omega-\rho|<1$ because $-\omega<\rho<2-\omega$.
        Therefore the integral \eqref{eq:def-omega-zeta} is absolutely convergent
        and independent of $\epsilon$.
\end{proof}

We define, by abuse of notation, the $\mathcal{C}$-linear map
$Z_{\omega}: \widehat{\mathfrak{H}^{0}} \to \mathbb{C}$ by
$Z_{\omega}(1)=1$ and \eqref{eq:def-omega-zeta}.
We call the value $Z_{\omega}(w)$ with $w \in \widehat{\mathfrak{H}^{0}}$
an \textit{$\omega$-deformed multiple zeta value} ($\omega$MZV).

\begin{prop}\cite[Theorem 4.5]{T-omega}
        For any $w, w' \in \widehat{\mathfrak{H}^{0}}$,
        it holds that $Z_{\omega}(w \ast_{\hbar} w')=Z_{\omega}(w)Z_{\omega}(w')$.
\end{prop}

Therefore, we can apply the results in the previous section to
$\omega$MZVs.

%%%%%%%%%%%%%%%%%%%%%%%%%%%%%%%%%

\subsection{Function $G_{s}(\omega)$ and its series expression}

Here we calculate $Z_{\omega}(\hbar^{-k}\phi_{k})$ for $k\ge 1$,
which is a corresponding object to the $q$-series
$G_{k}(q)=\sum_{n\ge 1}\sigma_{k-1}(n)q^{n}$.

\begin{prop}\label{prop:G-omega}
        For $k\ge 1$, it holds that
        \begin{align*}
                Z_{\omega}(\hbar^{-k}\phi_{k})=
                \int_{-\epsilon+i\,\mathbb{R}}dt\,
                \frac{(-t)^{k-1}}{(e^{2\pi i t}-1)(e^{-2\pi i \omega t}-1)}.
        \end{align*}
\end{prop}

To show Proposition \ref{prop:G-omega},
we use the following results.

\begin{prop}\cite[Proposition 3.5]{T-omega}\label{prop:omega-MZV-explicit}
        Let $r$ be a positive integer and suppose that
        $0<\epsilon<\min{\{1, 1/r\omega\}}$.
        It holds that
        \begin{align*}
                 &
                Z_{\omega}((e_{1}-g_{1})^{\alpha_{1}}g_{\beta_{1}+1} \cdots
                (e_{1}-g_{1})^{\alpha_{r}}g_{\beta_{r}+1})        \\
                 & =\prod_{a=1}^{r}\int_{-\epsilon+i\,\mathbb{R}}
                \frac{dt_{a}}{e^{2\pi i t_{a}}-1}
                \prod_{a=1}^{r}\left\{
                (-2\pi i \omega)^{\alpha_{a}}
                \binom{t_{a}+\alpha_{a}}{\alpha_{a}}
                \left(\frac{2\pi i \omega \, e^{2\pi i \omega (t_{1}+\cdots +t_{a})}}
                {1-e^{2\pi i \omega (t_{1}+\cdots +t_{a})}}
                \right)^{\beta_{a}+1}
                \right\}
        \end{align*}
        for any non-negative integers $\alpha_{1}, \ldots , \alpha_{r}$
        and $\beta_{1}, \ldots , \beta_{r}$.
\end{prop}

\begin{prop}\cite[Theorem 4.3]{T-omega}\label{prop:omega-duality}
        For any $r\ge 1$ and $\alpha_{1}, \ldots , \alpha_{r}, \beta_{1}, \ldots , \beta_{r}\ge 0$,
        it holds that
        \begin{align*}
                Z_{\omega}((e_{1}-g_{1})^{\alpha_{1}}g_{\beta_{1}+1} \cdots
                (e_{1}-g_{1})^{\alpha_{r}}g_{\beta_{r}+1})=
                Z_{\omega}((e_{1}-g_{1})^{\beta_{r}}g_{\alpha_{r}+1} \cdots
                (e_{1}-g_{1})^{\beta_{1}}g_{\alpha_{1}+1}).
        \end{align*}
\end{prop}

\begin{proof}[Proof of Proposition \ref{prop:G-omega}]
        Proposition \ref{prop:omega-duality} implies that
        $Z_{\omega}(g_{j})=Z_{\omega}((e_{1}-g_{1})^{j-1}g_{1})$ for $j\ge 1$.
        Therefore, it follows from Proposition \ref{prop:omega-MZV-explicit} that
        \begin{align*}
                Z_{\omega}(\hbar^{-k}\phi_{k}) & =
                \sum_{j=1}^{k}(j-1)!{k \brace j}
                (2\pi i \omega)^{-j}Z_{\omega}((e_{1}-g_{1})^{j-1}g_{1}) \\
                                               & =
                \int_{-\epsilon+i\,\mathbb{R}}
                \frac{dt}{(e^{2\pi i t}-1)(e^{-2\pi i \omega t}-1)}
                \sum_{j=1}^{k}(j-1)!{k \brace j}
                (-1)^{j-1}\binom{t+j-1}{j-1}.
        \end{align*}
        By using \eqref{eq:stirling-duality} and \eqref{eq:shifted-factorial},
        we see that
        \begin{align*}
                 &
                \sum_{j=1}^{k}(j-1)!{k \brace j}
                (-1)^{j-1}\binom{t+j-1}{j-1}
                =\sum_{j=1}^{k}(-1)^{j-1}{k \brace j}
                \prod_{a=1}^{j-1}(t+a)
                \\
                 & =\sum_{l=0}^{k-1}t^{l}\sum_{j=l+1}^{k}
                (-1)^{j-1}
                {k \brace j}{j \brack l+1}=(-t)^{k-1}.
        \end{align*}
        Thus we obtain the desired formula.
\end{proof}

Motivated by Proposition \ref{prop:G-omega},
we introduce the function
\begin{align}
        G_{s}(\omega)=\int_{C} dt\,
        \frac{(-t)^{s-1}}{(e^{2\pi i t}-1)(e^{-2\pi i \omega t}-1)}
        \label{eq:def-G-omega}
\end{align}
for a parameter $s \in \mathbb{C}$,
where $(-t)^{s-1}=e^{(s-1)\mathrm{Log}{(-t)}}$
is the principal value.
The contour $C$ is defined to be a rotation about $t=0$ of
the line $-\epsilon+i\mathbb{R}$ with $0<\epsilon<\min\{1,1/\omega\}$
which separates the poles at
$\mathbb{Z}_{\ge 0} \cup (\omega^{-1}\mathbb{Z}_{\ge 0})$ from those at
$\mathbb{Z}_{\le -1} \cup (\omega^{-1}\mathbb{Z}_{\le -1})$.
The integral is independent of the rotation angle because
the integrand behaves as $e^{-c|t|}$ with some $c>0$ as $|t|\to +\infty$
in any closed angle with vertex at $t=0$ contained in
the region $\mathbb{C}\setminus(\mathbb{R} \cup \omega^{-1}\mathbb{R})$.

The function $G_{s}(\omega)$ is holomorphic on
the cut plane $\mathbb{C}'=\mathbb{C}\setminus\mathbb{R}_{\le 0}$.
It has the following series expression.

\begin{prop}\label{prop:G-series}
        If $\mathrm{Re}s>2$, it holds that
        \begin{align}
                G_{s}(\omega)=(2\pi)^{-s}\Gamma(s)\left(
                e^{-s\pi i/2}
                \sum_{m\ge 0, \, n\ge 1}(m+n\omega)^{-s}-
                e^{s\pi i/2}
                \sum_{m\ge 1, \, n\ge 0}(m+n\omega)^{-s}
                \right).
                \label{eq:G-series}
        \end{align}
        If $s=2$, we have
        \begin{align}
                G_{2}(\omega)=(2\pi i)^{-2}\zeta(2)\left(\omega^{-2}-1\right)-
                \frac{1}{4\pi i \omega}.
                \label{eq:G-at-two}
        \end{align}
\end{prop}

\begin{proof}
        We may assume that $\omega>0$.

        First we consider the case where $\mathrm{Re}s>2$.
        Set $C_{\epsilon, \pm}=\{ix \mid \pm x \ge \epsilon\}$
        and $C_{\epsilon, 0}=\{\epsilon e^{i\theta} \mid \pi/2\le \theta\le 3\pi/2\}$.
        We deform the contour $C=-\epsilon+i\,\mathbb{R}$
        to $(C_{\epsilon, -})^{-1}\cup (C_{\epsilon, 0})^{-1} \cup C_{\epsilon, +}$.
        We denote the integrand of the right-hand side of \eqref{eq:def-G-omega}
        by $f(t)$.
        In the limit as $\epsilon \to +0$, it holds that
        \begin{align*}
                \int_{C_{\epsilon, +}}f(t)\,dt
                 & =
                \int_{\epsilon}^{\infty} idx\,(-ix)^{s-1}
                \frac{-e^{-2\pi \omega x}}{(1-e^{-2\pi x})(1-e^{-2\pi \omega x})} \\
                 & \to
                e^{-s\pi i/2}
                \int_{0}^{\infty} dx\,x^{s-1}
                \frac{e^{-2\pi \omega x}}{(1-e^{-2\pi x})(1-e^{-2\pi \omega x})}.
        \end{align*}
        since $\mathrm{Re}s>2$.
        It is equal to
        \begin{align*}
                e^{-s\pi i/2}\sum_{m\ge 0, \, n\ge 1}
                \int_{0}^{\infty} dx\,x^{s-1}e^{-2\pi (m+n\omega)x}=
                (2\pi)^{-s}\Gamma(s)e^{-s\pi i/2}\sum_{m\ge 0, \, n\ge 1} (m+n\omega)^{-s}.
        \end{align*}
        Similarly we find that
        \begin{align*}
                \int_{C_{\epsilon, -}}f(t)\,dt \to
                -(2\pi)^{-s}\Gamma(s)e^{s\pi i/2}\sum_{m\ge 1, \, n\ge 0} (m+n\omega)^{-s}
        \end{align*}
        as $\epsilon \to +0$ by setting $t=-ix$.
        Since $|f(t)|=O(|t|^{\mathrm{Re}s-3})$ as $t\to 0$,
        the integral $\int_{C_{\epsilon, 0}}f(t)\,dt$ vanishes
        in the limit as $\epsilon \to +0$ if $\mathrm{Re}s>2$.
        Thus we get \eqref{eq:G-series}.

        Next we consider the case of $s=2$.
        Then
        \begin{align*}
                \left(\int_{C_{\epsilon, +}}+\int_{(C_{\epsilon, -})^{-1}}\right)f(t)\,dt=
                \left(\int_{\epsilon}^{\infty}+\int_{-\infty}^{-\epsilon}\right)dx\,
                \frac{x}{(e^{-2\pi x}-1)(e^{2\pi \omega x}-1)}.
        \end{align*}
        By changing the variable $x$ to $-x$ in the second term, we get
        \begin{align*}
                 &
                \int_{\epsilon}^{\infty}dx\, \left(
                \frac{x}{(e^{-2\pi x}-1)(e^{2\pi \omega x}-1)}-\frac{x}{(e^{2\pi x}-1)(e^{-2\pi \omega x}-1)}
                \right)                                   \\
                 & =\int_{\epsilon}^{\infty}dx\, x \left(
                \frac{1}{e^{2\pi x}-1}-\frac{1}{e^{2\pi \omega x}-1}
                \right)                                   \\
                 & \to
                \int_{0}^{\infty}\frac{x}{e^{2\pi x}-1}dx-
                \int_{0}^{\infty}\frac{x}{e^{2\pi \omega x}-1}dx
                =\left(1-\omega^{-2}\right) (2\pi)^{-2}\zeta(2)
        \end{align*}
        in the limit as $\epsilon \to +0$.
        Since $f(t)=-(2\pi i)^{-2}\omega^{-1}t^{-1}+O(1)$ as $t \to 0$,
        it holds that
        \begin{align*}
                \int_{(C_{\epsilon, 0})^{-1}}f(t)\, dt \to -\pi i
                \left(-\frac{1}{(2\pi i)^{2}\omega}\right)=
                \frac{1}{4\pi i \omega}
                \qquad (\epsilon \to +0).
        \end{align*}
        Thus we obtain \eqref{eq:G-at-two}.
\end{proof}

%%%%%%%%%%%%%%%%%%%%%%%

\subsection{Three-term relation}

The function $G_{s}(\omega)$ satisfies the following three-term relation.

\begin{prop}\label{prop:G-omega-periodlike}
        On the cut plane $\mathbb{C}'=\mathbb{C}\setminus\mathbb{R}_{\le 0}$
        it holds that
        \begin{align}
                G_{s}(\omega)=G_{s}(\omega +1)+(\omega +1)^{-s}
                G_{s}(\frac{\omega}{\omega +1}).
                \label{eq:G-three-term}
        \end{align}
\end{prop}

\begin{proof}
        It suffices to show the equality in the case where $\omega>0$.
        Since
        \begin{align*}
                \frac{1}{e^{-2\pi i \omega t}-1}-
                \frac{1}{e^{-2\pi i (\omega+1) t}-1}=
                \frac{e^{2\pi i t}-1}
                {(e^{-2\pi i \omega t}-1)(e^{2\pi i (\omega+1) t}-1)},
        \end{align*}
        we see that
        \begin{align*}
                G_{s}(\omega)-G_{s}(\omega+1)=\int_{-\epsilon+i\,\mathbb{R}}\,dt\,
                \frac{(-t)^{s-1}}{(e^{2\pi i (\omega+1)t}-1)(e^{-2\pi i \omega t}-1)}=
                (\omega+1)^{-s}G_{s}(\frac{\omega}{\omega+1})
        \end{align*}
        by changing the variable $t \to t/(\omega+1)$.
\end{proof}

A function satisfying the three-term relation \eqref{eq:G-three-term}
is called a \textit{periodlike function}.
{}From Proposition \ref{prop:G-omega-periodlike},
$G_{s}(\omega)$ is a periodlike function.
In \cite{LZ} the following examples of periodlike function
are given.

\begin{example}\label{ex:LZ}
        \begin{enumerate}
                \item The function $\psi^{-}(\omega)=1-\omega^{-s}$ is a periodlike function
                      for any $s \in \mathbb{C}$ which is holomorphic in $\mathbb{C}'$.
                \item For $s \in \mathbb{C}$ satisfying $\mathrm{Re}s>2$,
                      we define the holomorphic function $\psi^{+}(\omega)$ in $\mathbb{C}'$ by
                      \begin{align*}
                              \psi^{+}(\omega)=\left(\sum_{m\ge 1, \, n\ge 0}+\sum_{m\ge 0, \, n\ge 1}\right)
                              (m+n\omega)^{-s}.
                      \end{align*}
                      Then $\psi^{+}(\omega)$ is a periodlike function.
        \end{enumerate}
\end{example}

On the other hand, from Proposition \ref{prop:G-series},
we obtain the following expression of $G_{s}(\omega)$
at the integer points $s=k\ge 3$.
Hence the function $G_{s}(\omega)$ interpolates the above two examples.

\begin{cor}\label{cor:G-integer}
        For any integer $k\ge 2$, we have
        \begin{align*}
                G_{k}(\omega)=\left\{
                \begin{array}{ll}
                        \displaystyle
                        \frac{(k-1)!}{(2\pi i )^{k}}\,\zeta(k)
                        \left(\omega^{-k}-1\right)-\frac{\delta_{k,2}}{4\pi i \omega}
                         & (\hbox{$k$:even}), \\
                        \displaystyle
                        \frac{(k-1)!}{(2\pi i)^{k}}
                        \left(
                        \sum_{m\ge 0, \, n\ge 1}+\sum_{m\ge 1, \, n\ge 0}
                        \right)
                        (m+n\omega)^{-k}
                         & (\hbox{$k$:odd}).
                \end{array}
                \right.
        \end{align*}
\end{cor}

%%%%%%%%%%%%%%%%%%%%%

\subsection{A generating series of $\omega$MZVs}

By applying the map $Z_{\omega}$ to both sides of
\eqref{eq:main-eN} or \eqref{eq:main-eN-explicit},
we obtain a formula for the generating series
of the $\omega$MZVs $Z_{\omega}((e_{N}^{\circ_{\hbar}L})^{r}) \, (r\ge 1)$
with fixed $N\ge 2$ and $L\ge 1$.
For simplicity, we write it only in the case of $N=2$.

It holds that
\begin{align}
        \log{\frac{\sin{\pi i x}}{\pi i x}}=
        \sum_{k\ge 1}\frac{(-1)^{k-1}}{k}\zeta(2k)x^{2k}=
        \sum_{n\ge 1}\frac{B_{2n}}{(2n)!}\frac{(2\pi x)^{2n}}{2n},
        \label{eq:bernoulli-even}
\end{align}
where $B_{n}$ is the Bernoulli number.
As an $\omega$-deformation of the Bernoulli number,
we define the set of polynomials $\{B_{2m}(\omega)\}_{m\ge 1}$
by the expansion
\begin{align}
        \log{\left(\frac{\sin{(\omega^{-1}\arcsin{(\pi i \omega x)})}}{\pi i x}\right)}=
        \sum_{n=1}^{\infty}\frac{B_{2n}(\omega)}{(2n)!}\frac{(2\pi x)^{2n}}{2n}.
        \label{eq:bernoulli-omega}
\end{align}
For example, we have
\begin{align*}
        B_{2}(\omega)
         & =-\frac{\omega^{2}-1}{6},                                                                             \\
        B_{4}(\omega)
         & =\frac{11\omega^{4}-10\omega^2-1}{30}=\frac{(\omega^2-1)(11\omega^2+1)}{30},                          \\
        B_{6}(\omega)
         & =-\frac{191\omega^6-168\omega^4-21\omega^2-2}{84}=-\frac{(\omega^2-1)(191\omega^4+23\omega^2+2)}{84}, \\
        B_{8}(\omega)
         & =\frac{2497\omega^8-2160\omega^6-294\omega^4-40\omega^2-3}{90}=
        \frac{(\omega^2-1)(11\omega^2+1)(227\omega^4+10\omega^2+3)}{90}.
\end{align*}
It follows from the definition that
$B_{2n}(\omega) \to B_{2n}$ in the limit as $\omega \to 0$.

\begin{thm}
        For $L\ge 1$, it holds that
        \begin{align*}
                1+\sum_{r\ge 1}(-1)^{(L-1)r}
                Z_{\omega}((e_{2}^{\circ_{\hbar} L})^{r})X^{r}=
                \exp{\left(L\sum_{n\ge 1}X^{n}
                        \left(-\frac{1}{\pi i \omega}\frac{(2\pi i \omega)^{2Ln}}{(Ln)^{2}\binom{2Ln}{Ln}}+
                        \frac{B_{2Ln}(\omega)}{(2Ln)!}\frac{(2\pi )^{2Ln}}{2Ln}
                        \right)
                        \right)}.
        \end{align*}
\end{thm}

\begin{proof}
        We apply the map $Z_{\omega}$ to both sides of \eqref{eq:main-e2},
        and change $X \to 2\pi i \omega X$.
        Then we obtain
        \begin{align}
                1+\sum_{r\ge 1}(-1)^{(L-1)r}Z_{\omega}((e_{2}^{\circ_{\hbar}L})^{r})X^{2Lr}=
                \exp{\left(\sum_{m=1}^{L}H(e^{m\pi i/L}X)\right)},
                \label{eq:e2-omega-generating}
        \end{align}
        where
        \begin{align*}
                H(X)=2\sum_{k\ge 1}
                \frac{(-1)^{k-1}}{(2k)!}Z_{\omega}(\hbar^{-2k}\phi_{2k})
                \sum_{m=1}^{L}\left(2\arcsin{(\pi i \omega X)}\right)^{2k}.
        \end{align*}
        It follows from Corollary \ref{cor:G-integer} and \eqref{eq:bernoulli-even}
        that
        \begin{align}
                H(X)
                 & =\sum_{k\ge 1}\frac{(-1)^{k-1}}{k}
                \left(\zeta(2k)(2\pi i)^{-2k}(\omega^{-2k}-1)-\frac{\delta_{k1}}{4\pi i \omega}\right)
                \left(2\arcsin{(\pi i \omega X)}\right)^{2k}
                \label{eq:formula-HX}
                \\
                 & =-\frac{\arcsin^{2}(\pi i \omega X)}{\pi i \omega}+
                \log{\frac{\sin{(\omega^{-1}\arcsin{(\pi i \omega X)})}}{\pi i X}}.
                \nonumber
        \end{align}
        By using \eqref{eq:arcsin-expansion} with $k=1$ and \eqref{eq:bernoulli-omega},
        we obtain
        \begin{align*}
                H(X)=-\frac{1}{\pi i \omega}
                \sum_{n\ge 1}\frac{(2\pi i \omega X)^{2n}}{n^2 \binom{2n}{n}}+
                \sum_{n\ge 1}\frac{B_{2n}(\omega)}{(2n)!}\frac{(2\pi X)^{2n}}{2n}.
        \end{align*}
        Now we get the desired equality by using
        \begin{align*}
                \sum_{m=1}^{L}(e^{m\pi i/L}X)^{2n}=\left\{
                \begin{array}{ll}
                        L\,X^{2n} & (L\mid n) \\ 0 & (\hbox{otherwise})
                \end{array}
                \right.
        \end{align*}
        and changing $X^{2L} \to X$.
\end{proof}

\begin{cor}\label{cor:omega-MZV-222}
        It holds that
        \begin{align*}
                1+\sum_{r\ge 1}Z_{\omega}(e_{2}^{r})X^{2r}=
                \frac{\sin{(\omega^{-1}\arcsin{(\pi i \omega X)})}}{\pi i X}
                \exp{\left(-\frac{\arcsin^{2}{(\pi i \omega X)}}{\pi i \omega}\right)}.
        \end{align*}
\end{cor}

\begin{proof}
        It follows from \eqref{eq:e2-omega-generating} with $L=1$ and
        \eqref{eq:formula-HX} because $H(-X)=H(X)$.
\end{proof}

%%%%%%%%%%%%%%%%%%%%%%%%%%%%%%%%%%%%%

\appendix

\section{Stirling numbers}\label{sec:app-stirling}

Stirling numbers ${m\brack n}$ and ${m\brace n}$
are defined by the initial condition
\begin{align*}
        {m\brack 0}={0 \brack m}=\delta_{m, 0}, \qquad
        {m\brace 0}={0\brace m}=\delta_{m, 0}
\end{align*}
and the recurrence relations
\begin{align}
        {m+1 \brack n }={m \brack n-1}+m{m \brack n},
        \qquad
        {m+1 \brace n}={m \brace n-1}+n{m \brace n}
        \label{eq:recurrence-stirling}
\end{align}
for $m, n \in \mathbb{Z}$.
It holds that ${m \brack n}={m \brace n}=0$ if $1\le m<n$.
In this paper we use the following relations.
See, e.g., \cite{AIK} for the proof.
\begin{align}
         &
        \frac{(e^{T}-1)^{m}}{m!}=\sum_{n\ge 1}
        {n \brace m }\frac{T^{n}}{n!} \qquad (m\ge 1),
        \label{eq:stirling-second-exp}
        \\
         &
        \frac{(-\log{(1-T)})^{m}}{m!}=\sum_{n\ge 1}
        {n \brack m}\frac{T^{n}}{n!} \qquad (m\ge 1),
        \label{eq:stirling-first-log}
        \\
         &
        \sum_{j\ge 0}(-1)^{j}{m \brace j} {j \brack n}=(-1)^{m}\delta_{m,n}
        \qquad (m, n\ge 0),
        \label{eq:stirling-duality}  \\
         &
        \prod_{a=1}^{m}(z+a)
        =\sum_{a=0}^{m}{m+1 \brack a+1} z^{a} \qquad (m\ge 0),
        \label{eq:shifted-factorial} \\
         &
        \sum_{l=0}^{n}(-1)^{l}\binom{n}{l}l^{m}=
        (-1)^{n}n!{ m \brace n} \qquad (m, n\ge 0).
        \label{eq:binom-power-sum}
\end{align}
{}From \eqref{eq:shifted-factorial}, we see that
\begin{align}
        { l \brack 1}=(l-1)!, \quad
        {l \brack k}=(l-1)!\sum_{l>m_{1}>\cdots >m_{k-1}>0}
        \frac{1}{m_{1} \cdots m_{k-1}} \qquad (l\ge 1, k\ge 2).
        \label{eq:stirling-first-harmonic}
\end{align}

%%%%%%%%%%%%%%%%%%%

\section{Properties of the power series $\varphi(\theta; x$)}\label{sec:app2}

\begin{lem}
        For $m\ge 0$ and $n\ge 1$, it holds that
        \begin{align*}
                \sum_{k=0}^{n}(-1)^{k}\binom{n}{k}(k+1)^{m}=
                (-1)^{n}n! { m+1 \brace n+1 }.
        \end{align*}
        In particular, if $n>m\ge 0$, then we have
        \begin{align}
                \sum_{k=0}^{n-1}(-1)^{k}\binom{n}{k}(k+1)^{m}=
                (-1)^{n-1}(n+1)^{m}.
                \label{eq:lem-power-sum}
        \end{align}
\end{lem}

\begin{proof}
        We have
        \begin{align*}
                \sum_{k=0}^{n}(-1)^{k}\binom{n}{k}(k+1)^{m} & =
                \sum_{k=0}^{n}(-1)^{k}
                \left(\binom{n+1}{k+1}-\binom{n}{k+1}\right)
                (k+1)^{m}                                                                                     \\
                                                            & =\sum_{k=0}^{n+1}(-1)^{k-1}\binom{n+1}{k}k^{m}-
                \sum_{k=0}^{n}(-1)^{k-1}\binom{n}{k}k^{m}.
        \end{align*}
        It follows from \eqref{eq:recurrence-stirling} and \eqref{eq:binom-power-sum}
        that the right-hand side is equal to
        \begin{align*}
                (-1)^{n}(n+1)! {m \brace n+1}-
                (-1)^{n-1}n! {m \brace n}=
                (-1)^{n}n! { m+1 \brace n+1}.
        \end{align*}
\end{proof}

\begin{lem}\label{lem:key-strange}
        For $n\ge 0$, it holds that
        \begin{align}
                \frac{1}{x-y}\left(\prod_{a=1}^{n}(x-a)-\prod_{a=1}^{n}(y-a)\right)=
                \sum_{k=0}^{n-1}\binom{n}{k}
                \prod_{a=1}^{k}\left(\frac{k+1}{n+1}y-a\right)
                \prod_{a=1}^{n-1-k}\left(x-\frac{k+1}{n+1}y-a\right).
                \label{eq:lem-key-strange}
        \end{align}
\end{lem}

\begin{proof}
        It is trivial when $n=0$. We assume that $n\ge 1$.
        Both sides are monic polynomials in $x$ of
        degree $n-1$.
        Hence it suffices to show that they are equal at
        $x=m$ for $1\le m \le n$.
        We denote the right-hand side by $P(x)$.

        Suppose that $1\le m \le n$.
        We calculate
        \begin{align*}
                P(m) & =\sum_{k=0}^{n-1}(-1)^{k}\binom{n}{k}
                \prod_{a=1}^{k}\left(-\frac{k+1}{n+1}y+a\right)
                \prod_{a=m-n+1+k}^{m-1}\left(-\frac{k+1}{n+1}y+a\right) \\
                     & =\sum_{k=0}^{n-1}(-1)^{k}\binom{n}{k}
                \prod_{a=1}^{m-1}\left(-\frac{k+1}{n+1}y+a\right)
                \prod_{a=m-n+1+k}^{k}\left(-\frac{k+1}{n+1}y+a\right)   \\
                     & =\sum_{k=0}^{n-1}(-1)^{k}\binom{n}{k}
                \prod_{a=1}^{m-1}\left(-\frac{k+1}{n+1}y+a\right)
                \prod_{a=1}^{n-m}\left((k+1)(1-\frac{y}{n+1})-a\right).
        \end{align*}
        By using \eqref{eq:shifted-factorial},
        we see that the right-hand side is equal to
        \begin{align*}
                 &
                \sum_{a=0}^{m-1}\sum_{b=0}^{n-m}(-1)^{n-m-b}
                {m \brack a+1} {n-m+1 \brack b+1}
                \left(-\frac{y}{n+1}\right)^{a}
                \left(1-\frac{y}{n+1}\right)^{b} \\
                 & {}\qquad \times
                \sum_{k=0}^{n-1}(-1)^{k}\binom{n}{k}(k+1)^{a+b}.
        \end{align*}
        By applying \eqref{eq:lem-power-sum} to the sum over $0\le k\le n-1$
        and using \eqref{eq:shifted-factorial} again,
        we obtain
        \begin{align*}
                 &
                \sum_{a=0}^{m-1}\sum_{b=0}^{n-m}(-1)^{n-m-b}
                {m \brack a+1}{n-m+1 \brack b+1}
                \left(-\frac{y}{n+1}\right)^{a}
                \left(1-\frac{y}{n+1}\right)^{b}(-1)^{n-1}(n+1)^{a+b} \\
                 & =(-1)^{n-1}
                \sum_{a=0}^{m-1}
                {m \brack a+1 }(-y)^{a}
                \sum_{b=0}^{n-m}
                (-1)^{n-m-b}
                {n-m+1 \brack b+1}
                (n+1-y)^{b}                                           \\
                 & =
                (-1)^{n-1}\prod_{a=1}^{m-1}(-y+a)\prod_{a=1}^{n-m}(n+1-y-a)=
                \prod_{a=1}^{m-1}(y-a)\prod_{a=m+1}^{n}(y-a),
        \end{align*}
        which is equal to the left-hand side of \eqref{eq:lem-key-strange}
        at $x=m$.
\end{proof}

\begin{lem}\label{lem:exp-strange}
        It holds that
        \begin{align}
                e^{\lambda \varphi(\theta; x)}=1+\lambda \sum_{n=1}^{\infty}\frac{x^{n}}{n!}
                \prod_{a=1}^{n-1}(\lambda+n\theta-a).
                \label{eq:exp-strange}
        \end{align}
\end{lem}

\begin{proof}
        Denote the left-hand side by $\Psi(x)$
        and set $\Psi(x)=\sum_{n\ge 0}d_{n}x^{n}$.
        The power series $\Psi(x)$ is uniquely determined from $\Psi(0)=1$ and
        the differential equation
        \begin{align*}
                \Psi'(x)=\lambda
                \left(\sum_{k\ge 0}\frac{x^{k+1}}{k!}\prod_{a=1}^{k}((k+1)\theta-a)\right)
                \Psi(x),
        \end{align*}
        which imply that $d_{0}=1$ and
        the recurrence relation
        \begin{align*}
                d_{n+1}=\frac{\lambda}{n+1}\sum_{k=0}^{n}\frac{d_{n-k}}{k!}
                \prod_{a=1}^{k}((k+1)\theta-a)
        \end{align*}
        for $n\ge 0$.
        We prove
        \begin{align*}
                d_{n}=\frac{\lambda}{n!}\prod_{a=1}^{n-1}(\lambda+n\theta-a)
        \end{align*}
        for $n\ge 1$ by induction on $n$.
        It holds when $n=1$ since $d_{1}=\lambda d_{0}=\lambda$.
        Suppose that $n\ge 1$.
                {}From the induction hypothesis, we see that
        \begin{align*}
                d_{n+1}=\frac{\lambda}{(n+1)!}\left(
                \prod_{a=1}^{n}((n+1)\theta-a)+\lambda
                \sum_{k=0}^{n-1}\binom{n}{k}
                \prod_{a=1}^{k}((k+1)\theta-a)
                \prod_{a=1}^{n-k-1}(\lambda+(n-k)\theta-a)
                \right).
        \end{align*}
        By using \eqref{eq:lem-key-strange} with
        $x=\lambda+(n+1)\theta$ and $y=(n+1)\theta$,
        we see that the sum over $0\le k \le n-1$ is equal to
        \begin{align*}
                \frac{1}{\lambda}\left(
                \prod_{a=1}^{n}(\lambda+(n+1)\theta-a)-\prod_{a=1}^{n}((n+1)\theta-a)
                \right).
        \end{align*}
        Thus we obtain
        \begin{align*}
                d_{n+1}=\frac{\lambda}{(n+1)!}\prod_{a=1}^{n}(\lambda+(n+1)\theta-a).
        \end{align*}
        This completes the proof.
\end{proof}

\begin{proof}[Proof of Proposition \ref{prop:varphi-power}]
        We expand both sides of \eqref{eq:exp-strange}
        into a power series of $\lambda$.
        We see that
        \begin{align*}
                 &
                \lambda \sum_{n\ge 1}\frac{x^{n}}{n!}
                \prod_{a=1}^{n-1}(\lambda+n\theta-a)=
                \lambda \sum_{n\ge 1}\frac{x^{n}}{n!}
                \sum_{j\ge 1}(-1)^{n-j}{n \brack j}
                (\lambda+n\theta)^{j-1}                                    \\
                 & =              \lambda \sum_{n\ge 1}\frac{x^{n}}{n!}
                \sum_{j\ge 1}(-1)^{n-j}{n \brack j}
                \sum_{k=1}^{j}\binom{j-1}{k-1}(n\theta)^{j-k}\lambda^{k-1} \\
                 & =\sum_{k\ge 1}\lambda^{k}
                \sum_{n\ge 1}\frac{x^{n}}{n!}
                \sum_{j\ge k}(-1)^{n-j}{n \brack j}
                \binom{j-1}{k-1}(n\theta)^{j-k}.
        \end{align*}
        Hence the desired equality holds.
\end{proof}

\begin{prop}
        It holds that
        \begin{align}
                e^{\varphi(\theta; x)}-xe^{\theta\varphi(\theta; x)}=1.
                \label{eq:varphi-rel}
        \end{align}
\end{prop}

\begin{proof}
        Lemma \ref{lem:exp-strange} implies that
        \begin{align*}
                 &
                e^{\varphi(\theta; x)}-xe^{\theta \varphi(\theta; x)} \\
                 & =1+\sum_{n\ge 1}\frac{x^{n}}{n!}
                \prod_{a=1}^{n-1}(1+n\theta-a)-x\left(
                1+\theta \sum_{n\ge 1} \frac{x^{n}}{n!}
                \prod_{a=1}^{n-1}((n+1)\theta-a)
                \right)                                               \\
                 & =1+x+\theta \sum_{n\ge 2}\frac{x^{n}}{(n-1)!}
                \prod_{a=1}^{n-2}(n\theta-a)-x \left(
                1+\theta \sum_{n\ge 2}\frac{x^{n-1}}{(n-1)!}
                \prod_{a=1}^{n-2}(n\theta-a)
                \right)=1.
        \end{align*}
\end{proof}

\begin{proof}[Proof of Proposition \ref{prop:varphi}]
        Set $\theta=1/N$ and change $x \to -x$ in
        \eqref{eq:varphi-rel}.
        {}From the obtained equation, we see that
        $\alpha(x)^{N}=x^{N}(1-\alpha(x))$.
        It implies \eqref{eq:alpha-factorize}
        as noted in Remark \ref{rem:existence-alpha}.
\end{proof}

%%%%%%%%%%%%%%%%%%%%%%

\end{document}